\documentclass[journal,twoside,web]{ieeecolor}

\usepackage{generic} 
\usepackage{cite}
\usepackage{graphicx}
\usepackage{textcomp}
\let\labelindent\relax 

\usepackage{mathtools}
\usepackage{amsmath,amssymb,amsfonts}
\usepackage{hyperref}
\usepackage{cancel}
\usepackage{array}

\usepackage{algorithm}
\usepackage{algpseudocode}
\algrenewcommand\algorithmicrequire{\textbf{Input:}}
\algrenewcommand\algorithmicensure{\textbf{Output:}}

\usepackage{subcaption}

\DeclareCaptionLabelSeparator{periodspace}{.\quad}
\usepackage[nameinlink,capitalise]{cleveref}
\crefname{equation}{}{} 

\usepackage{standalone}
\usepackage{breakcites}
\usepackage{enumitem}
\newlist{identities}{enumerate}{1}
\setlist[identities,1]{
    label={\textbf{ID \arabic*}},
    ref=\arabic*, 
    wide,itemsep=0pt,topsep=0pt}
\creflabelformat{identitiesi}{#2\textup{#1}#3} 
\crefname{identitiesi}{ID}{IDs} 
\Crefname{identitiesi}{ID}{IDs}

\usepackage{multicol}

\usepackage{tikz}
\usetikzlibrary{shapes,arrows,positioning,calc}

\usepackage{pgfplots}
\pgfplotsset{compat=1.18}
\usepgfplotslibrary{colorbrewer} 

\usepackage{pgfplotstable}
\usepackage{booktabs}
\usepackage{colortbl}

\newtheorem{theorem}{Theorem}
\newtheorem{definition}{Definition}
\newtheorem{proposition}{Proposition}
\newtheorem{remark}{Remark}
\newtheorem{corollary}{Corollary}
\newtheorem{lemma}{Lemma}

\crefname{lemma}{Lemma}{Lemmas} 
\Crefname{lemma}{Lemma}{Lemmas}

\newcommand*\circled[1]{
\tikz[baseline=(char.base)]{\node(char)[shape=rounded rectangle,draw,inner sep=0.6pt,minimum height=1.5ex]{#1};}
} %
\newcommand\kronF[2]{{#1}^{\circled{\tiny{\ensuremath{#2}}}}} 
\renewcommand\Vec[1][]{\textnormal{vec}\ensuremath{\if$#1$ \else \left[#1\right]\fi}}

\newcommand{\real}{\mathbb{R}}

\newcommand{\rd}{\text{\upshape d}} 

\newcommand{\bzero}{\ensuremath{\mathbf{0}}} 

\newcommand{\tv}{\ensuremath{\tilde{v}}}
\newcommand{\tw}{\ensuremath{\tilde{w}}}

\newcommand{\bA}{\ensuremath{\mathbf{A}}}
\newcommand{\bB}{\ensuremath{\mathbf{B}}}
\newcommand{\bC}{\ensuremath{\mathbf{C}}}
\newcommand{\bD}{\ensuremath{\mathbf{D}}}

\newcommand{\bF}{\ensuremath{\mathbf{F}}}
\newcommand{\bG}{\ensuremath{\mathbf{G}}}
\newcommand{\bH}{\ensuremath{\mathbf{H}}}
\newcommand{\bI}{\ensuremath{\mathbf{I}}}

\newcommand{\bP}{\ensuremath{\mathbf{P}}}
\newcommand{\bQ}{\ensuremath{\mathbf{Q}}}

\newcommand{\bS}{\ensuremath{\mathbf{S}}}

\newcommand{\bV}{\ensuremath{\mathbf{V}}}
\newcommand{\bW}{\ensuremath{\mathbf{W}}}

\newcommand{\ba}{\ensuremath{\mathbf{a}}}

\renewcommand{\bf}{\ensuremath{\mathbf{f}}}
\newcommand{\bg}{\ensuremath{\mathbf{g}}}
\newcommand{\bh}{\ensuremath{\mathbf{h}}}

\newcommand{\bu}{\ensuremath{\mathbf{u}}}
\newcommand{\bv}{\ensuremath{\mathbf{v}}}
\newcommand{\bw}{\ensuremath{\mathbf{w}}}
\newcommand{\bx}{\ensuremath{\mathbf{x}}}
\newcommand{\by}{\ensuremath{\mathbf{y}}}

\newcommand{\cE}{\ensuremath{\mathcal{E}}}

\newcommand{\cH}{\ensuremath{\mathcal{H}}}

\newcommand{\cL}{\ensuremath{\mathcal{L}}}

\newcommand{\perm}[2]{\ensuremath{\bS_{#1 \times #2}}}

\newcommand{\G}{\ensuremath{\mathbf{G}}}
\newcommand{\Hx}{\ensuremath{\mathbf{H}}}

\def\BibTeX{{\rm B\kern-.05em{\sc i\kern-.025em b}\kern-.08em
T\kern-.1667em\lower.7ex\hbox{E}\kern-.125emX}}
\begin{document}
\title{Scalable Computation of $\cH_\infty$ Energy Functions for Polynomial Control-Affine Systems}
\author{Nicholas A. Corbin, Boris Kramer
    \thanks{This work was supported by the National Science Foundation under Grant CMMI-2130727.}
    \thanks{N. Corbin and B. Kramer are with the Department of Mechanical and Aerospace Engineering, University of California San Diego, La Jolla, CA 92093-0411 USA (e-mail: ncorbin@ucsd.edu, bmkramer@ucsd.edu).}
}

\maketitle

\begin{abstract}
    We present a scalable approach to computing nonlinear balancing energy functions for control-affine systems with polynomial nonlinearities.
    Al'brekht's power-series method is used to solve the Hamilton-Jacobi-Bellman equations for polynomial approximations to the energy functions.
    The contribution of this article lies in the numerical implementation of the method based on the Kronecker product, enabling scalability to over 1000 state dimensions.
    The tensor structure and symmetries arising from the Kronecker product representation are key to the development of efficient and scalable algorithms.
    We derive the explicit algebraic structure for the equations, present rigorous theory for the solvability and algorithmic complexity of those equations, and provide general purpose open-source software implementations for the proposed algorithms.
    The method is illustrated on two simple academic models, followed by a high-dimensional semidiscretized PDE model of dimension as large as $n=1080$.
\end{abstract}



\section{Introduction}\label{sec:intro}
Mathematical models are used throughout engineering design, analysis, and control.
Often though, significant effort has to be expended to find a compromise between accuracy and computational complexity.
Model reduction methods seek to systematically develop low-dimensional surrogate models
so that the surrogates are much faster to simulate, yet retain important characteristics of the
high-dimensional models.

For control applications, balanced truncation (BT) model reduction \cite{Mullis1976,Moore1981} is popular since it is based on
controllability and observability.
Balanced truncation has been widely studied for linear time-invariant (LTI) systems, for which many variations and formulations of BT exist \cite{Gugercin2004}.
Its success for LTI systems has stimulated much interest in extending BT to nonlinear systems.
While Scherpen provided the theoretical extensions to nonlinear control-affine systems \cite{Scherpen1993,Scherpen1994a,Scherpen1994,Scherpen1996}, \emph{scalable} computational methods to implement nonlinear BT
have remained an open area of research.
Consequently, nonlinear BT has yet to be
demonstrated
on any high-dimensional systems, i.e. those requiring model reduction.

The two main challenges in nonlinear BT are 1) solving the Hamilton-Jacobi-Bellman (HJB) partial differential equations (PDEs) for the controllability and observability-type nonlinear \emph{energy functions}, and 2) computing a nonlinear coordinate transformation that ``balances'' these energy functions.
Fujimoto \& Tsubakino \cite{Fujimoto2008a} and Krener \cite{Krener2008} independently showed that a Taylor series approach can form the foundation for a computational framework tackling these two challenges.
However, neither of these works addressed the scalability issues that arise or provided a numerical solution, and as a result the approach is only illustrated on low-dimensional academic examples with $n=4$ and $n=6$ degrees of freedom, respectively.

Other numerical methods for solving HJB PDEs include traditional grid-based discretization techniques \cite{Falcone2016}, policy/value iteration approaches (including Galerkin approximations \cite{Beard1997,Kalise2018,Dolgov2021}), and other iterative methods such as sum-of-squares optimization \cite{Parrilo2000} and machine learning \cite{Borovykh2022}.
Due to the curse of dimensionality, many of these approaches do not scale well to dimensions above about $n=10$.
Efforts have been made to scale into the 100s of dimensions \cite{Dolgov2021}, but ensuring convergence to the correct solution is typically a significant challenge.
Some of these methods---when they converge---can solve for viscosity solutions to HJB PDEs, whereas the Taylor series approach only works for sufficiently smooth systems.
Fortunately, nonlinear BT theory is only defined for sufficiently smooth systems.
Among all existing methods for solving HJB PDEs for the energy functions, only the Taylor series approach has been shown to be suitable for subsequently computing nonlinear balancing transformations \cite{Fujimoto2008a,Krener2008}.

Another approach involves avoiding solving the full HJB PDE altogether.
One notable example is the state-dependent Riccati equation (SDRE) method~\cite{Cimen2012}, which
involves factoring the dynamics into a ``linear-like'' structure with state-dependent system matrices.
The Riccati equation---now also state-dependent---can then be solved in lieu of the HJB PDE.
In practice, the SDRE is also often solved using Taylor expansions, so this process amounts to discarding certain terms in the HJB PDE.
Finally, nonlinear BT has been treated with the algebraic Gramian approach \cite{Condon2005,Gray2006,Benner2024}; however, this method produces quadratic energy functions which fail to capture many features of the inherently non-quadratic energy functions exhibited by nonlinear systems, so there is a concern that algebraic Gramians are too conservative.
In this work, we seek to expand the capabilities of solving the full HJB PDE using the Taylor series approach \emph{without} these simplifications.

The Taylor series approach to solving HJB PDEs,
referred to as Al'brekht's method, dates back the origins of modern control theory in the 1960s \cite{Albrekht1961,Lukes1969}.
Despite frequent use in the literature
\cite{Garrard1972,Scherpen1994a,Fujimoto2010}, Al'brekht's method has traditionally only been applied to models with a few degrees of freedom or limited nonlinearities.
The primary deterrent has been the computational complexity of the approach:
without efficient implementation and solvers, the method scales very poorly beyond a few degrees of freedom.
Additionally,
besides Krener's Nonlinear Systems Toolbox (NST) \cite{Krener2019}, there has been a lack of general-purpose software for using Al'brekht's method, and unfortunately symbolic computations used in NST hinder its scalability.
The last few years have seen renewed interest in Al'brekht method \cite{Breiten2018,Almubarak2019,Borggaard2020,Borggaard2021},
in part due to the introduction of novel high-performance solvers
adapted to the tensor structure of the equations arising in Al'brekht's method \cite{Chen2019,Borggaard2021}.
One of these recent works includes Kramer et al.~\cite{Kramer2024}, which computes nonlinear BT energy functions for systems with state dimension as large as $n=1024$.
However, that work assumes very limited nonlinearity in the form of quadratic drift, linear inputs, and linear outputs.

A scalable computational approach for \emph{general} high-dimensional polynomial control-affine systems has remained an open problem due to the difficulty of forming and solving the tensor equations for general polynomial systems.
The present paper provides a solution to this problem.

This article contains two main contributions.
The first is a scalable Kronecker product-based approach to computing energy function approximations for systems with general polynomial structure in the drift, input, and measurements.
To that end, we derive the explicit equations for the energy function coefficients.
Second, we provide rigorous theoretical analyses regarding solvability and computational complexity, along with numerics demonstrating the scalability of the approach.
Open-access software implementations for the proposed algorithms are available in the \texttt{cnick1/NLBalancing} repository \cite{NLBalancing2023} under the \texttt{v1.0.0} tag, along with all of the numerical examples.

This paper is structured as follows.
\Cref{sec:background} reviews preliminary notation and definitions.
The proposed algorithm for computing $\cH_\infty$ energy function approximations is presented in \cref{sec:NLBT-Poly}, along with solvability and scalability analyses.
Numerical results are presented in \cref{sec:results} and \cref{sec:fem} to demonstrate the accuracy, convergence, and scalability of the proposed method.
Finally, \cref{sec:conclusion} gives a summary and future directions for the work.

\section{Preliminaries and Background}\label{sec:background}
In \cref{sec:notation}, basic notation and definitions relating to Kronecker product polynomial expansions are reviewed.
Afterwards, we review the definitions for the $\cH_\infty$ nonlinear~BT energy functions in \cref{sec:NLBT-background}.

\subsection{Notation and Kronecker Product Identities}\label{sec:notation}
The Kronecker product of two matrices $\bA \in \real^{p \times q}$ and $\bB \in \real^{s \times t}$ is the $ps \times qt$ block matrix
\begin{align*}
    \bA \otimes \bB \coloneqq \begin{bmatrix} a_{11}\bB & \cdots & a_{1q}\bB \\
                \vdots    & \ddots & \vdots    \\
                a_{p1}\bB & \cdots & a_{pq}\bB
                              \end{bmatrix},
\end{align*}
where $a_{ij}$ denotes the $(i,j)$th entry of $\bA$.
Repeated Kronecker products are written as $\kronF{\bx}{k} \coloneqq \underbrace{\bx \otimes \dots \otimes \bx}_{k \ \text{times}} \in \real^{n^k}.$
The $\Vec{[\cdot]}$ operator stacks the columns of a matrix into one tall column vector, and the \textit{perfect shuffle matrix} \cite{VanLoan2000} is defined as the permutation matrix which shuffles $\Vec[\bA]$ to match $\Vec[\bA^\top]$:
\begin{equation}
    \Vec[\bA^\top]=\perm{q}{p} \Vec[\bA].
\end{equation}
For $\bA \in \real^{p \times q}$, the
\textit{$k$-way Lyapunov matrix} is defined as
\small
\begin{equation}
    \label{eq:kWayLyapunov}
    \cL_k(\bA) \coloneqq \sum_{i=1}^k\underbrace{\bI_p \otimes \bA \otimes \bI_p \otimes \dots \otimes \bI_p}_{\text{$k$ factors, $\bA$ in the $i$th position}} \in \real^{p^k \times p^{k-1}q}.
\end{equation}
\normalsize

\cref{tab:dimensions} provides a collection of Kronecker product identities compiled from various sources \cite{Brewer1978,VanLoan2000,Henderson1981,Magnus2019}.
\begin{table}[htb]
    \caption{Relevant Kronecker product identities. }
    \-\hspace{-2pt}\makebox[\columnwidth]{\rule{.95\columnwidth}{0.1em}}
    \vspace{-7pt}

    \-\hspace{-2pt}\makebox[\columnwidth]{\rule{.95\columnwidth}{0.1em}}

    \begin{identities}
        \setlength{\itemsep}{3pt}
        \setlength\itemindent{14pt}
        \item \label{ID:T2.4} $\quad (\bA \otimes \bB)(\bD \otimes \bG) = \bA \bD \otimes \bB \bG $
        \item \label{ID:T2.5} $\quad \bA \otimes \bB = \perm{s}{p} (\bB \otimes \bA)\perm{q}{t}$
        \item \label{ID:T2.17} $\quad (\bI_p \otimes \bx)\bA = \bA \otimes \bx$
        \item \label{ID:T2.13} $\quad \Vec[\bA \bD \bB] = (\bB^\top \otimes \bA) \Vec[\bD]$
        \item \label{ID:T3.4} $\quad \begin{aligned}[t]
                \Vec[\bA \bD] & = (\bI_s \otimes \bA) \Vec[\bD] \\
            \end{aligned}$
        \item \label{ID:nonsymmetric-quadraticForm} $\quad \bu^\top \bB \bx  = \Vec[\bB]^\top (\bx \otimes \bu) $
        \item \label{ID:vecKron1} $\quad \Vec\left[\bx^\top \otimes \bI_m\right] = (\bx \otimes \Vec[\bI_m])$
        \item \label{ID:vecKron3} $\quad \Vec\left[\bA \otimes \bB\right] =  \left(\bI_q \otimes \perm{p}{t} \otimes \bI_s \right)\left(\Vec[\bA] \otimes \Vec[\bB] \right)$
    \end{identities}

    \vspace{5pt}

    \-\hspace{9pt}\textit{Dimensions of matrices used in the Kronecker product identities}\\

    \vspace{-2pt}
    \centering
    \begin{tabular}{c}
        $\bA(p \times q)$ \,\,
        $\bD(q \times s)$ \,\,
        $\bB(s \times t)$ \,\,
        $\bG(t \times u)$ \,\,
        $\bu(s \times 1)$ \,\,
        $\bx(t \times 1)$
    \end{tabular}

    \vspace{2pt}

    \-\hspace{-2pt}\makebox[\columnwidth]{\rule{.95\columnwidth}{0.1em}}
    \label{tab:dimensions}
\end{table}

\noindent A concept which arises when dealing with polynomials in Kronecker product form is symmetry of the coefficients (a generalization of symmetry of a matrix), which is defined next.
\begin{definition}[Symmetric Coefficients\label{def:sym}] Given a monomial of the form $\bw_d^\top \kronF{\bx}{d}$, the coefficient $\bw_k \in \real^{n^k \times 1}$ is \emph{symmetric} if for all $\ba_i \in \real^n$ it satisfies
    \begin{displaymath}
        \bw_k^\top \left(\ba_1 \otimes \ba_2 \otimes \cdots \otimes \ba_k\right) = \bw_k^\top \left(\ba_{i_1} \otimes \ba_{i_2} \otimes \cdots \otimes \ba_{i_k}\right),
    \end{displaymath}
    where the indices $\{ i_j \}_{j=1}^k$ are any permutation of $\{1, \dots, k \}$.
\end{definition}

A symmetric coefficient is thus invariant under certain permutations; this can also be represented in terms of the perfect shuffle matrix.
\begin{proposition}[Permutation of symmetric coefficients]\label{thm:symPermutation}
    If a coefficient $\bw_k \in \real^{n^k \times 1}$ is symmetric as per \cref{def:sym}, then
    \begin{align*}
        \bw_k & = \perm{n^j}{n^i} \bw_k \qquad \forall i,j\geq0 \quad \text{s.t.}\quad  i+j=k.
    \end{align*}
\end{proposition}

\subsection{Energy Functions for \texorpdfstring{$\cH_\infty$}{H-infinity} Nonlinear Balancing}\label{sec:NLBT-background}
Consider
the control-affine dynamical system
\begin{align}\label{eq:FOM-NL}
    \dot{\bx}(t) & = \bf(\bx(t))  + \bg(\bx(t)) \bu(t), &
    \by(t)       & = \bh(\bx(t)),
\end{align}
with $m$ inputs, $p$ outputs, and state dimension $n$.
The $\cH_\infty$ nonlinear balancing framework \cite{Scherpen1996} defines a pair of energy functions that generalize the concepts of controllability and observability to (potentially unstable) systems of the form \cref{eq:FOM-NL}.
These energy functions are then balanced using a nonlinear state-space transformation, and subsequent model reduction involves truncating states that are determined to be less important in the balanced representation.
However, computing these energy functions (defined next) is a significant challenge, which we address in this work.

\begin{definition}\cite[Def. 5.1]{Scherpen1996}
    Let $\gamma$ be a positive constant $\gamma > 0, \gamma \neq 1$, and define $\eta \coloneqq 1-\gamma^{-2}$.
    The $\cH_\infty$ past energy of the nonlinear system \cref{eq:FOM-NL} is defined as
    \begin{equation} \label{eq:HinftyPastEnergyDef}
        \cE_\gamma^{-}(\bx_0)  \coloneqq \!\!\!\! \min_{\substack{\bu \in L_{2}(-\infty, 0] \\ \bx(-\infty) = \bzero,\,  \bx(0) = \bx_0}} \! \frac{1}{2} \int\displaylimits_{-\infty}^{0} \eta \Vert \by(t) \Vert^2  +  \Vert \bu(t) \Vert^2 {\rm{d}}t.
    \end{equation}
    The $\cH_\infty$ future energy
    is defined as
    \begin{equation} \label{eq:HinftyFutureEnergyDef-a}
        \cE_\gamma^{+}(\bx_0)  \coloneqq  \underset{\substack{\bu \in L_{2}[0,\infty) \\ \bx(0) = \bx_0 , \, \bx(\infty) = \bzero}}{{\min}/{\max}} \! \frac{1}{2} \int\displaylimits_{0}^{\infty} \Vert \by(t) \Vert^2  +
        \frac{\Vert \bu(t) \Vert^2}{\eta} {\rm{d}}t,
    \end{equation}
    where the minimum is taken for $\gamma>1$ and the maximum is taken for $\gamma<1$.
\end{definition}

The energy functions, which are nominally defined by optimization problems, can be computed as the solutions to HJB PDEs \cite[Thm. 5.2]{Scherpen1996}.
Assume that the HJB equation
\begin{equation} \label{eq:Hinfty-Past-HJB}
    \begin{split}
        0 & =  \frac{\partial \cE_\gamma^{-}(\bx)}{\partial \bx} \bf(\bx) + \frac{1}{2}  \frac{\partial \cE_\gamma^{-}(\bx)}{\partial \bx} \bg(\bx) \bg(\bx)^\top \frac{\partial^\top \cE_\gamma^{-}(\bx)}{\partial \bx} \\
          & \qquad - \frac{\eta}{2}  \bh(\bx)^\top  \bh(\bx)
    \end{split}
\end{equation}
has a solution with $\cE_\gamma^{-}(\bzero) = 0$ such that the quantity
$ - \bf(\bx) - \bg(\bx) \bg(\bx)^\top \partial^\top \cE_\gamma^{-}(\bx)/\partial \bx $
is asymptotically stable.
Then this solution is the past energy function $\cE_\gamma^{-}(\bx)$ from~\cref{eq:HinftyPastEnergyDef}.
Furthermore, assume that the HJB equation
\begin{equation} \label{eq:Hinfty-Future-HJB}
    \begin{split}
        0 & =  \frac{\partial \cE_\gamma^{+}(\bx)}{\partial \bx} \bf(\bx)   - \frac{\eta}{2} \frac{\partial \cE_\gamma^{+}(\bx)}{\partial \bx} \bg(\bx) \bg(\bx)^\top \frac{\partial^\top \cE_\gamma^{+}(\bx)}{\partial \bx} \\
          & \qquad + \frac{1}{2}\bh(\bx)^\top \bh(\bx)
    \end{split}
\end{equation}
has a solution with $\cE_\gamma^{+}(\bzero) = 0$ such that the quantity
$\left.\bf(\bx) - \eta \bg(\bx) \bg(\bx)^\top \partial^\top \cE_\gamma^{+}(\bx)/\partial \bx\right.$
is asymptotically stable.
This solution is the future energy function $\cE_\gamma^{+}(\bx)$ from~\cref{eq:HinftyFutureEnergyDef-a}.

\begin{remark}
    We adopt the $\cH_\infty$ balancing framework because it generalizes the open-loop \cite{Scherpen1993} and closed-loop HJB \cite{Scherpen1994} balancing theories.
    Under appropriate assumptions about existence and smoothness of the energy functions,
    the closed-loop HJB past and future energy functions are recovered
    in the limit as the gain parameter $\gamma$ goes to infinity (i.e. $\eta$ goes to one),
    whereas
    the open-loop nonlinear controllability and observability energy functions are recovered
    as the parameter $\gamma$ goes to one (i.e. $\eta$ goes to zero)
    \cite[Thm.~5.5~\&~5.7]{Scherpen1996}.
\end{remark}

\section{Computing \texorpdfstring{$\cH_\infty$}{H-infinity} Energy Functions for Polynomial Control-affine Systems} \label{sec:NLBT-Poly}
Computing solutions to the HJB equations \cref{eq:Hinfty-Past-HJB,eq:Hinfty-Future-HJB} in general is very challenging.
However, if $\bf(\bx)$, $\bg(\bx)$, and $\bh(\bx)$ are analytic, the solutions to \cref{eq:Hinfty-Past-HJB,eq:Hinfty-Future-HJB} are known to be analytic as well \cite{Lukes1969}.
Al'brekht showed that in this case, it is possible to compute the Taylor expansion of the energy functions $\cE_\gamma^{-}(\bx)$ and $\cE_\gamma^{+}(\bx)$ based on the Taylor expansions of $\bf(\bx)$, $\bg(\bx)$, and $\bh(\bx)$.
Thus, for the rest of this paper, we will consider a nonlinear control-affine dynamical system with polynomial structure
\begin{equation}\label{eq:FOM-Poly}
    \begin{split}
        \dot{\bx} & = \underbrace{\bA \bx + \sum_{p=2}^\ell \bF_p \kronF{\bx}{p}}_{\bf(\bx)} + \underbrace{\left(\sum_{p=1}^\ell \G_p \left(\kronF{\bx}{p} \otimes \bI_m\right) + \bB\right)}_{\bg(\bx)} \bu, \\
        \by       & = \underbrace{\bC \bx + \sum_{p=2}^\ell \Hx_p \kronF{\bx}{p}}_{\bh(\bx)},
    \end{split}
\end{equation}
where $\bA \in \real^{n\times n}$, $\bF_p \in \real^{n\times n^p}$, $\bB \in \real^{n\times m}$, $\G_p \in \real^{n \times mn^p}$, $\bC \in \real^{p\times n}$, and $\Hx_p \in \real^{p\times n^p}$.
We emphasize that many common nonlinear dynamical systems can be put in polynomial form \cref{eq:FOM-Poly}.
One can consider \cref{eq:FOM-Poly} as simply a Taylor approximation to the control-affine system \cref{eq:FOM-NL}.
Furthermore, many common PDEs, including Navier-Stokes, Kuramoto-Sivashinsky, Burgers, Allen-Cahn, Korteweg-de Vries, and Fokker-Planck all feature polynomial nonlinearities;
upon spatial discretization, these all yield systems of the form \cref{eq:FOM-Poly}.

\subsection{Main Results: Energy Function Approximations for Polynomial Systems}\label{sec:mainResult}
Since the energy function solutions to \cref{eq:Hinfty-Past-HJB,eq:Hinfty-Future-HJB} are analytic, they
can be approximated as $d$th-order polynomials
\begin{align}
    \label{eq:vi-coeffs}
     & \cE_\gamma^-(\bx)
    \approx \frac{1}{2} \sum_{i=2}^d \bv_i^\top \kronF{\bx}{i},
     &
     & \cE_\gamma^+(\bx)
    \approx \frac{1}{2} \sum_{i=2}^d \bw_i^\top \kronF{\bx}{i},
\end{align}
with coefficients $\bv_i, \bw_i \in \real^{n^i}$.
Note that
the first term in the sum can be written $\bv_2^\top \kronF{\bx}{2} = \bx^\top \bV_2 \bx$, and without loss of generality we can assume $\bV_2$ and $\bW_2$ are symmetric.
The next two theorems give the explicit equations to compute the polynomial coefficients $\bv_i$ and $\bw_i$.
\begin{theorem}[Past energy polynomial coefficients]\label{thm:viPoly}
    Let $\gamma > \gamma_0 \geq 0$ and $\eta = 1-\gamma^{-2}$, where $\gamma_0$ denotes the smallest $\tilde{\gamma}$ such that a stabilizing controller exists for which the $\mathcal{H}_\infty$ norm of the closed-loop system is less than $\tilde{\gamma}$.
    Let the past energy function $\cE_\gamma^{-}(\bx)$, which solves the $\cH_\infty$ HJB PDE \cref{eq:Hinfty-Past-HJB} for the polynomial system \cref{eq:FOM-Poly}, be of the form~\cref{eq:vi-coeffs} with the coefficients $\bv_i \in \real^{n^i}$ for $i=2,3,\dots,d$.
    Then $\bv_2 = \Vec\left[\bV_2\right]$, where $\bV_2$ is the symmetric positive semidefinite solution to the $\cH_\infty$ algebraic Riccati equation (ARE)
    \begin{align}\label{eq:V2-2}
        \bzero & = \bA^\top \bV_2 + \bV_2 \bA- \eta\bC^\top \bC + \bV_2 \bB \bB^\top \bV_2.
    \end{align}
    For $3 \leq k \leq d$, let $\mathbf{\tv}_k \in \real^{n^k}$ solve the linear system
    \small
    \begin{equation}\label{eq:LinSysForVk-Poly}
        \begin{split}
             & \cL_{k} \left(\bA + \bB \bB^\top  \bV_2 \right)^\top \mathbf{\tv}_k
            =
            -\sum_{\mathclap{\substack{i,p\geq 2                                                                                            \\       i + p = k+1}}}
            \cL_i(\bF_p)^\top \bv_i
            \\
             & \hspace{.4cm}
            - \frac{1}{4}\sum_{\mathclap{\substack{i,j>2                                                                                    \\ i+j=k+2}}} ij~\Vec[\bV_i^\top \bB \bB^\top \bV_j]
            + \eta\sum_{\mathclap{\substack{p,q\geq 1                                                                                       \\ p + q = k}}} \Vec[\Hx_p^\top \Hx_q] \\
             & \hspace{.4cm}- \frac{1}{4} \sum_{o=1}^{2\ell}\left( \sum_{\substack{p,q \geq 0                                               \\p+q=o}} \left(  \sum_{\substack{i,j\geq 2 \\       i+j=k-o+2}} \!\!\!\!\!\!ij~\Vec \Biggl[ \left(\bI_{n^p} \otimes \Vec\left[ \bI_m\right]^\top\right) \times \right.\right.\\
             & \hspace{1cm} \left(\Vec\left[\G_q^\top \bV_j \right]^\top \otimes \left(\G_p^\top \bV_i \otimes \bI_m \right) \right) \times \\
             & \hspace{1cm} \left.\left. \left(\bI_{n^{j-1}} \otimes \perm{n^{i-1}}{n^q m} \otimes \bI_m \right)
                    \left( \bI_{n^{k-p}} \otimes \Vec\left[ \bI_m\right] \right)\Biggr]
            \vphantom{\sum_{\substack{p,q \geq 0                                                                                            \\p+q=o}}} \right)\right),
        \end{split}
    \end{equation}
    \normalsize
    where
    $\Hx_1 \coloneqq \bC$ and
    $\G_0 \coloneqq \bB$.
    Then the coefficient $\bv_k \in \real^{n^k}$ for $3 \leq k \leq d$ is obtained by symmetrization of $\mathbf{\tv}_k$.
\end{theorem}
\begin{theorem}[Future energy polynomial coefficients]\label{thm:wiPoly}
    Let $\gamma > \gamma_0 \geq 0$ and $\eta = 1- \gamma^{-2}$ as in \cref{thm:viPoly}.
    Let the future energy function $\cE_\gamma^{+}(\bx)$, which solves the $\cH_\infty$ HJB PDE \cref{eq:Hinfty-Future-HJB} for the polynomial system \cref{eq:FOM-Poly}, be of the form~\cref{eq:vi-coeffs} with the coefficients $\bw_i \in \real^{n^i}$ for $i=2,3,\dots,d$.
    Then $\bw_2 = \Vec\left[\bW_2\right]$, where $\bW_2$ is the symmetric positive semidefinite solution to the $\cH_\infty$ ARE
    \begin{align}
        \bzero & = \bA^\top \bW_2 + \bW_2 \bA + \bC^\top \bC - \eta \bW_2 \bB \bB^\top \bW_2.\label{eq:W2-2}
    \end{align}
    For $3 \leq k \leq d$, let $\mathbf{\tw}_k \in \real^{n^k}$ solve the linear system
    \small
    \begin{equation}\label{eq:LinSysForWk-Poly}
        \begin{split}
             &
            \cL_{k} \left(\bA - \eta \bB \bB^\top  \bW_2 \right)^\top
            \mathbf{\tw}_k =
            - \sum_{\mathclap{\substack{i,p\geq 2                                                                                           \\ i + p = k+1}}} \cL_i(\bF_p)^\top \bw_i
            \\ & \hspace{.4cm}
            + \frac{\eta}{4}\sum_{\mathclap{\substack{i,j>2                                                                                 \\ i+j=k+2}}} ij~\Vec[\bW_i^\top \bB \bB^\top \bW_j]
            - \sum_{\mathclap{\substack{p,q\geq 1                                                                                           \\ p + q = k}}} \Vec[\Hx_p^\top \Hx_q] \\
             & \hspace{.4cm} + \frac{\eta}{4} \sum_{o=1}^{2\ell}\left( \sum_{\substack{p,q \geq 0                                           \\p+q=o}} \left(  \sum_{\substack{i,j\geq 2 \\       i+j=k-o+2}} \!\!\!\!\!\!ij~\Vec \Biggl[ \left(\bI_{n^p} \otimes \Vec\left[ \bI_m\right]^\top\right) \times \right.\right.\\
             & \hspace{1cm} \left(\Vec\left[\G_q^\top \bW_j \right]^\top \otimes \left(\G_p^\top \bW_i \otimes \bI_m \right) \right) \times \\
             & \hspace{1cm} \left.\left.\left(\bI_{n^{j-1}} \otimes \perm{n^{i-1}}{n^q m} \otimes \bI_m \right)
                    \left( \bI_{n^{k-p}} \otimes \Vec\left[ \bI_m\right] \right)\Biggr]
            \vphantom{\sum_{\substack{p,q \geq 0                                                                                            \\p+q=o}}} \right)\right).
        \end{split}
    \end{equation}
    \normalsize
    Then the coefficient $\bw_k \in \real^{n^k}$ for $3 \leq k \leq d$ is obtained by symmetrization of $\mathbf{\tw}_k$.
\end{theorem}

Before proving \cref{thm:viPoly,thm:wiPoly}, a few intermediate results are necessary.
Inserting the polynomial forms of
$\bf(\bx)$, $\bg(\bx)$, and $\bh(\bx)$ from~\cref{eq:FOM-Poly}
into the HJB PDE~\cref{eq:Hinfty-Past-HJB} gives
\begin{subequations}\label{eq:Hinfty-Past-HJB-plugged-in}
    \begin{align}
        0  = & \frac{\partial \cE_{\gamma}^{-}(\bx)}{\partial \bx} \left[\bA \bx + \sum_{p=2}^\ell \bF_p \kronF{\bx}{p}\right]  \label{eq:Hinfty-Past-HJB-plugged-in-1}                                              \\
             & + \frac{1}{2}  \frac{\partial \cE_{\gamma}^{-}(\bx)}{\partial \bx} \left[\sum_{p=1}^\ell \G_p \left(\kronF{\bx}{p} \otimes \bI_m \right) + \bB\right] \times  \label{eq:Hinfty-Past-HJB-plugged-in-2} \\
             & \left[\sum_{p=1}^\ell  \left({\kronF{\bx}{p}}^\top \otimes \bI_m \right)\G_p^\top + \bB^\top\right] \frac{\partial^\top \cE_{\gamma}^{-}(\bx)}{\partial \bx} \label{eq:Hinfty-Past-HJB-plugged-in-3}  \\
             & - \frac{\eta}{2}\left[\bC \bx + \sum_{p=2}^\ell \Hx_p \kronF{\bx}{p}\right]^\top \left[\bC \bx + \sum_{p=2}^\ell \Hx_p \kronF{\bx}{p}\right] \label{eq:Hinfty-Past-HJB-plugged-in-4},
    \end{align}
\end{subequations}
where by the product rule, the gradient of the past energy function \cref{eq:vi-coeffs} in Kronecker product form~is
\begin{equation}\label{eq:Past-energy-deriv}
    \begin{split}
         & \frac{\partial \cE_{\gamma}^{-}(\bx)}{\partial \bx}
        = \frac{1}{2}\left (2\bx^\top\bV_2  \right.                                                                                                        \\
         & + \bv_3^\top (\bI_n \otimes \bx \otimes \bx) +  \bv_3^\top (\bx \otimes \bI_n \otimes \bx) + \bv_3^\top (\bx \otimes \bx \otimes \bI_n)         \\
         & + \left.\bv_4^\top (\bI_n \otimes \bx \otimes \bx \otimes \bx) + \bv_4^\top ( \bx \otimes \bI_n \otimes \bx \otimes \bx)     + \cdots \right ).
    \end{split}
\end{equation}
Since $\partial \cE_{\gamma}^{-}(\bx)/\partial \bx$ is known explicitly, the HJB PDE no longer contains derivatives, making it an algebraic equation rather than a differential equation.
Collecting terms of the same degree in $\bx$ gives a separate equation for each coefficient $\bv_k$.
However, writing the HJB PDEs \cref{eq:Hinfty-Past-HJB,eq:Hinfty-Future-HJB} with the explicit Kronecker product forms for $\bf$, $\bg$, $\bh$, $\cE_{\gamma}^{-}$, and $\cE_{\gamma}^{+}$ is an arduous task, let alone when the multiplication of all the polynomials in \cref{eq:Hinfty-Past-HJB-plugged-in} is expanded.
The next few \lcnamecrefs{thm:albrekht} and \lcnamecrefs{thm:Past-energy-deriv-compact} are therefore introduced to aid in writing things more compactly.
We focus on the results only for the past energy function coefficients $\bv_i$ for brevity; analogous results hold for the future energy coefficients $\bw_i$.
\begin{lemma}\label[lemma]{thm:albrekht}
    After plugging in polynomial expressions for the dynamics~\cref{eq:FOM-Poly} and the energy functions~\cref{eq:vi-coeffs}, the HJB PDE \cref{eq:Hinfty-Past-HJB} gives a set of equations for each coefficient $\bv_i$ for $i=2,3,\dots,d$.
    The equation for the degree~2 coefficient
    $\bv_2$
    is a \emph{quadratic} algebraic equation equivalent to the ARE \cref{eq:V2-2}.
\end{lemma}
\begin{proof}
    The proof is centered around counting factors of $\bx$ that appear in each term.
    Observe that the lowest-order terms in the gradient of the energy function in \cref{eq:Past-energy-deriv} are order~1, meaning that order~2 terms can only contain one more factor of $\bx$.
    Hence the collection of degree~2 terms in \cref{eq:Hinfty-Past-HJB-plugged-in} is
    \begin{equation*}
        \begin{split}
            0  = & \bx^\top \bV_2  \bA \bx   + \frac{1}{2} \bx^\top \bV_2 \bB  \bB^\top \bV_2\bx  - \frac{\eta}{2} \bx^\top \bC^\top \bC \bx.
        \end{split}
    \end{equation*}
    Differentiating twice with respect to $\bx$ reveals the ARE \cref{eq:V2-2}.
\end{proof}
\begin{lemma}\label[lemma]{thm:albrekht-2}
    The equations for the remaining coefficients $\bv_k$ and $\bw_k$ for $3 \leq k \leq d$ are \emph{linear} algebraic equations.
\end{lemma}
\begin{proof}
    The equation for the $k$th coefficient $\bv_k$ comes from collecting the degree $k$ terms in \cref{eq:Hinfty-Past-HJB-plugged-in}.
    There are only two types of degree~$k$ terms containing $\bv_k$: the terms containing $\bA$ and the terms containing $\bB \bB^\top$, which also contain $\bv_2$.
    All of the other terms with $\bv_k$ are at least degree $k+1$, as can be verified by counting how many factors of $\bx$ they contain.
    Since the degree~$k$ terms containing $\bv_k$ only contain one factor of $\bv_k$, the resulting algebraic equation is linear.
\end{proof}

\begin{lemma}\label{thm:LHS}
    The coefficient matrix for the
    equation for $\bv_k$ for $3 \leq k \leq d$ has the form $\cL_{k} \left(\bA + \bB \bB^\top  \bV_2 \right)^\top$.
\end{lemma}
\begin{proof}
    The collection of degree~$k$ terms
    in \cref{eq:Hinfty-Past-HJB-plugged-in}
    is
    \begin{align*}
         & \frac{1}{2} \bv_k^\top\left( (  \bI_n \otimes  \kronF{\bx}{k-1}) +  ( \bx \otimes \bI_n \otimes  \kronF{\bx}{k-2})+ \cdots \right) \times \\
         & \qquad\left(\bA + \bB\bB^\top \bV_2 \right) \bx = (\text{everything else})\kronF{\bx}{k}.
    \end{align*}
    Here, we have separated the terms on the left of the equals sign as the terms containing $\bv_k$, and the terms on the right are all of the remaining terms of degree $k$ that do not contain $\bv_k$;
    they will be derived later.
    Using the Kronecker product identities in \cref{tab:dimensions}, one can verify that this is equivalent to
    \begin{align}
        \frac{1}{2} \bv_k^\top\cL_k \left(\bA + \bB\bB^\top \bV_2 \right) \kronF{\bx}{k} = (\text{everything else})\kronF{\bx}{k}.\label{eq:kth-terms-sofar-2}
    \end{align}
    Requiring this to hold for all $\bx$ and transposing leads to a linear system for $\bv_k$ with coefficient matrix $\cL_{k} \left(\bA + \bB \bB^\top  \bV_2 \right)^\top$.
\end{proof}

\cref{thm:LHS} proves the left-hand-side of the linear system \cref{eq:LinSysForVk-Poly}; what remains is to prove the right-hand-side, which consists of the contributions due to $\bF_p$, $\G_p$, and $\Hx_p$.
The following two observations help to simplify the remaining derivations.
\begin{corollary}\label{thm:G-RHS}
    The nonlinear contributions in the dynamics, namely
    $\bF_p$, $\G_p$, and $\Hx_p$, only appear on the right-hand sides of the
    equations for the coefficients $\bv_k$ for $3 \leq k \leq d$.
\end{corollary}
\begin{corollary}\label{thm:Past-energy-deriv-compact}
    When computing the $k$th coefficient of the energy function $\bv_k$,
    the terms containing $\bF_p$ and $\G_p$ only appear with coefficients $\bv_2,\dots,\bv_{k-1}$.
    These coefficients have already been computed, so they are symmetric by construction.
\end{corollary}

Leveraging this symmetry, we rewrite the gradient of the energy function~\cref{eq:Past-energy-deriv}
using \cref{ID:T2.5,thm:symPermutation}~as
\begin{align}
    \frac{\partial \cE_{\gamma}^{-}(\bx)}{\partial \bx}
     & = \frac{1}{2}\left (2\bv_2^\top (\bI_n \otimes \bx)  + 3\bv_3^\top (\bI_n \otimes \bx \otimes \bx)
    + \cdots \right ) \nonumber                                                                                         \\
     & = \frac{1}{2}\sum_{i=2}^{k-1} i \bv_i^\top (\bI_n \otimes \kronF{\bx}{i-1}).\label{eq:Past-energy-deriv-compact}
\end{align}
Note that we truncate the terms above $i=k-1$ since the coefficients above $\bv_{k-1}$ do not enter the analysis regarding contributions from the inclusion of $\bF_p$, $\G_p$, and $\Hx_p$.

Now we are ready to finish the proof of \cref{thm:viPoly}.
In particular, we focus on proving the terms containing contributions from $\bF_p$, $\G_p$, and $\Hx_p$,
which corresponds to the right-hand-side of \cref{eq:LinSysForVk-Poly}, or the terms labeled ``(everything else)'' in \cref{eq:kth-terms-sofar-2}.

\noindent\hspace{2em}{\itshape Proof of \cref{thm:viPoly}: }
Upon inserting the polynomial expansions for $\bf(\bx)$, $\bg(\bx)$, and $\bh(\bx)$ from \cref{eq:FOM-Poly}, along with the gradient of the energy function \cref{eq:Past-energy-deriv-compact} exploiting the symmetry from~\cref{thm:Past-energy-deriv-compact}, the HJB PDE \cref{eq:Hinfty-Past-HJB} becomes
\small
\begin{subequations}\label{eq:Hinfty-Past-HJB-plugged-in-symmetric}
    \begin{align}
         & 0  =  \frac{1}{2} \left[\sum_{i=2}^{k-1} i \bv_i^\top (\bI_n \otimes \kronF{\bx}{i-1})\right] \left[\bA \bx + \sum_{p=2}^\ell \bF_p \kronF{\bx}{p}\right]  \label{eq:Hinfty-Past-HJB-plugged-in-symmetric-1}                           \\
         & + \frac{1}{8} \left[\sum_{i=2}^{k-1} i \bv_i^\top (\bI_n \otimes \kronF{\bx}{i-1})\right] \left[\sum_{p=1}^\ell \G_p \left(\kronF{\bx}{p} \otimes \bI_m \right) + \bB\right] \times  \label{eq:Hinfty-Past-HJB-plugged-in-symmetric-2} \\
         & \left[\sum_{q=1}^\ell  \left({\kronF{\bx}{q}}^\top \otimes \bI_m \right)\G_q^\top + \bB^\top\right] \left[\sum_{j=2}^{k-1} (\bI_n \otimes {\kronF{\bx}{j-1}}^\top) \bv_j j\right] \label{eq:Hinfty-Past-HJB-plugged-in-symmetric-3}    \\
         & - \frac{\eta}{2}\left[\bx^\top \bC^\top + \sum_{p=2}^\ell {\kronF{\bx}{p}}^\top \Hx_p^\top \right] \left[\bC \bx + \sum_{q=2}^\ell \Hx_q \kronF{\bx}{q}\right] \label{eq:Hinfty-Past-HJB-plugged-in-symmetric-4}.
    \end{align}
\end{subequations}
\normalsize
The summation indices $i$, $j$, $p$, and $q$ are introduced to keep track of the number of factors of $\bx$ when we expand the multiplication of these polynomials and collect terms of the same degree.
From the collection of degree~2 terms, \cref{thm:albrekht} shows that $\bV_2$ solves ARE \cref{eq:V2-2}.
For $3 \leq k \leq d$, the collection of degree~$k$ terms leads to a linear algebraic equation for $\bv_k$; \cref{thm:LHS} shows that the coefficient matrix for the linear system is as in \cref{eq:LinSysForVk-Poly}.
Therefore we need to prove the additional terms on the right-hand-side of \cref{eq:LinSysForVk-Poly}, starting with the set of terms containing $\bF_p$.
An arbitrary $k$th-order term from \cref{eq:Hinfty-Past-HJB-plugged-in-symmetric-1} containing $\bF_p$ is
\begin{align*}
     & \frac{1}{2}i \bv_i^\top (\bI_n \otimes \kronF{\bx}{i-1}) \bF_p \kronF{\bx}{p}, &  & \text{with} &  & p + i - 1 = k.
\end{align*}
The quantity $\bF_p \kronF{\bx}{p}$ has dimension $n \times 1$, so we apply \cref{ID:T2.17} and then \cref{ID:T2.4} to combine the factors of $\bx$ to rewrite this as
$\frac{1}{2}i \bv_i^\top (\bF_p  \otimes \bI_{n^{i-1}}) \kronF{\bx}{k}$.
The multiplication by $i$ is expanded into a sum of $i$ terms; then, since $\bv_i$ and $\kronF{\bx}{k}$ are symmetric as in \cref{def:sym}, \cref{thm:symPermutation} allows us to permute the quantities on the right, leading to the definition of the $i$-way Lyapunov matrix:
\begin{align}
    \frac{1}{2}i \bv_i^\top (\bF_p  \otimes \bI_{n^{i-1}})\kronF{\bx}{k} & = \frac{1}{2}\bv_i^\top (\bF_p  \otimes \bI_n \otimes \dots  \otimes \bI_n)\kronF{\bx}{k}  \nonumber \\
                                                                         & + \frac{1}{2}\bv_i^\top (\bI_n \otimes \bF_p  \otimes \dots  \otimes \bI_n)\kronF{\bx}{k}  \nonumber \\
    + \dots                                                              & + \frac{1}{2}\bv_i^\top (   \bI_n \otimes \dots \otimes \bI_n \otimes \bF_p)\kronF{\bx}{k} \nonumber \\
                                                                         & = \frac{1}{2}\bv_i^\top \cL_i(\bF_p) \kronF{\bx}{k}.\label{eq:N-terms}
\end{align}

Moving on, we write an arbitrary $k$th-order term from \cref{eq:Hinfty-Past-HJB-plugged-in-symmetric-2,eq:Hinfty-Past-HJB-plugged-in-symmetric-3} containing $\G_p$ as
\small
\begin{align}\label{eq:vi-G-terms-1}
     & \frac{1}{8} i \bv_i^\top (\bI_n \otimes \kronF{\bx}{i-1}) \G_p (\kronF{\bx}{p} \otimes \bI_m)  \times                     \\
     & \hspace*{2.5cm} ({\kronF{\bx}{q}}^\top \otimes \bI_m)\G_q^\top (\bI_n \otimes {\kronF{\bx}{j-1}}^\top) \bv_j j, \nonumber
\end{align}
\normalsize
with $p \in \left[0, o\right]$, $o \in \left[1,2\ell\right]$,
$q = o - p$, and $ i + j + o = k + 2$.
For now, we drop the $1/8$ factor for readability.
Due to the symmetry of $\bv_i$ and $\bv_j$,
the factors $\bv_i^\top (\bI_n \otimes \kronF{\bx}{i-1})$ and $(\bI_n \otimes {\kronF{\bx}{j-1}}^\top) \bv_j$
can be simplified with \cref{ID:T3.4} to rewrite \cref{{eq:vi-G-terms-1}} as
\begin{align*}
     & ij~{\kronF{\bx}{i-1}}^\top  \bV_i^\top \left[ \G_p (\kronF{\bx}{p} \otimes \bI_m)({\kronF{\bx}{q}}^\top \otimes \bI_m)\G_q^\top \right] \bV_j {\kronF{\bx}{j-1}}.
\end{align*}
From here, applying \cref{ID:nonsymmetric-quadraticForm} combines the outer $\bx$ factors to give
\small
\begin{align}
     & = ij~\Vec \left[\left[ \bV_i^\top \G_p (\kronF{\bx}{p} \otimes \bI_m)\right]({\kronF{\bx}{q}}^\top \otimes \bI_m) \left[\G_q^\top \bV_j \right]\right]^\top {\kronF{\bx}{\underset{\!-2}{i+j}}}. \label{eq:vi-G-terms-2}
\end{align}
\normalsize
Note where we placed extra brackets grouping factors in the $\Vec[\cdot]^\top$ portion to apply \cref{ID:T2.13}, which leads to
\small
\begin{align} \label{eq:vi-G-terms-3}
     & = ij~\Vec\left[  {\kronF{\bx}{q}}^\top \otimes \bI_m\right]^\top \left(\G_q^\top  \bV_j \otimes ({\kronF{\bx}{p}}^\top \otimes \bI_m) \G_p^\top \bV_i\right) {\kronF{\bx}{\underset{\!-2}{i+j}}}.
\end{align}
\normalsize
Since this whole quantity is a scalar, we enclose it in $\Vec{[\cdot]}$, apply \cref{ID:T2.13} again, and transpose the result to obtain
\begin{equation}
    \begin{split} \label{eq:vi-G-terms-4}
         & = ij~\Vec\left[\G_q^\top  \bV_j \otimes ({\kronF{\bx}{p}}^\top \otimes \bI_m) \G_p^\top \bV_i \right]^\top \times \\
         & \hspace{3cm}\left({\kronF{\bx}{i+j-2}} \otimes \Vec\left[  {\kronF{\bx}{q}}^\top \otimes \bI_m\right] \right).
    \end{split}
\end{equation}
The factor ${\kronF{\bx}{q}}^\top$ is extracted using \cref{ID:vecKron1}, leading to
\begin{equation}
    \begin{split} \label{eq:vi-G-terms-5}
         & = ij~\Vec\left[\G_q^\top  \bV_j \otimes ({\kronF{\bx}{p}}^\top \otimes \bI_m) \G_p^\top \bV_i \right]^\top \times \\
         & \hspace{4cm} \left({\kronF{\bx}{i+j+q-2}} \otimes \Vec\left[ \bI_m\right] \right).
    \end{split}
\end{equation}
Noting that $i+j+q-2 = k-p$ and applying \cref{ID:T2.17}, we can pull out the factor of ${\kronF{\bx}{i+j+q-2}}={\kronF{\bx}{k-p}}$ to reach
\begin{equation}
    \begin{split} \label{eq:vi-G-terms-6}
         & = ij~\Vec\left[\G_q^\top  \bV_j \otimes ({\kronF{\bx}{p}}^\top \otimes \bI_m) \G_p^\top \bV_i \right]^\top \times \\
         & \hspace{3.75cm}\left( \bI_{n^{k-p}} \otimes \Vec\left[ \bI_m\right] \right) {\kronF{\bx}{k-p}}.
    \end{split}
\end{equation}
To extract the remaining $\kronF{\bx}{p}$ from the $\Vec[\cdot]^\top$ factor, we need to apply \cref{ID:vecKron3}.
Isolating this $\Vec[\cdot]$ factor (without the transpose) while we make these simplifications, \cref{ID:vecKron3} gives
\begin{equation}\label{eq:vi-G-terms-6-2}
    \begin{split}
         & \Vec\left[\G_q^\top  \bV_j \otimes ({\kronF{\bx}{p}}^\top \otimes \bI_m) \G_p^\top \bV_i \right]                                           \\
         & \quad = \bP \left(\Vec\left[\G_q^\top \bV_j \right] \otimes \Vec\left[({\kronF{\bx}{p}}^\top \otimes \bI_m)\G_p^\top \bV_i \right]\right),
    \end{split}
\end{equation}
where we denote the special permutation matrix from \cref{ID:vecKron3} as $\bP~\coloneqq~\left(\bI_{n^{j-1}} \otimes \perm{n^q m}{n^{i-1}} \otimes \bI_m \right)$ for readability.
The factor {\small$\Vec\left[({\kronF{\bx}{p}}^\top \otimes \bI_m)\G_p^\top \bV_i \right]$} at the end can be manipulated using the second form of \cref{ID:T3.4}
and \cref{ID:vecKron1} into
    {\small$\left(\bV_i^\top \G_p \otimes \bI_m \right) \left(\kronF{\bx}{p} \otimes \Vec\left[ \bI_m\right]\right)$}.
Plugging this back into \cref{eq:vi-G-terms-6-2}, the $\Vec[\cdot]$ factor becomes
\small
\begin{align*}
     & = \bP \left(\Vec\left[\G_q^\top \bV_j \right] (1) \otimes \left(\left(\bV_i^\top \G_p \otimes \bI_m \right) \left(\kronF{\bx}{p} \otimes \Vec\left[ \bI_m\right]\right)\right)\right).
\end{align*}
\normalsize
Introducing the factor of 1
in the second line enables a subtle but critical step; it
allows us to pull the factor containing $\kronF{\bx}{p}$ out from within the nested products with \cref{ID:T2.4}:
\small
\begin{align*}
     & = \bP \left(\Vec\left[\G_q^\top \bV_j \right] \otimes
    \left(\bV_i^\top \G_p \otimes \bI_m \right) \right)
    \left(\cancel{1 \otimes} \kronF{\bx}{p} \otimes \Vec\left[ \bI_m\right]\right) .
\end{align*}
\normalsize
The $\kronF{\bx}{p}$ factor is extracted at last with \cref{ID:T2.17} to obtain
\small
\begin{align*}
     & = \bP \left(\Vec\left[\G_q^\top \bV_j \right] \otimes
    \left(\bV_i^\top \G_p \otimes \bI_m \right) \right)
    \left(\bI_{n^p} \otimes \Vec\left[ \bI_m\right]\right) \kronF{\bx}{p} .
\end{align*}
\normalsize
Transposing this entire quantity and inserting it back into the HJB term \cref{eq:vi-G-terms-6} gives
\small
\begin{align}
     & ij~{\kronF{\bx}{p}}^\top \left(\bI_{n^p} \otimes \Vec\left[ \bI_m\right]^\top\right)   \times \label{eq:vi-G-terms-7} \\
     & \left(\Vec\left[\G_q^\top \bV_j \right]^\top \otimes \left(\G_p^\top \bV_i \otimes \bI_m \right) \right)\bP^\top
    \left( \bI_{n^{k-p}} \otimes \Vec\left[ \bI_m\right] \right) {\kronF{\bx}{k-p}}. \nonumber
\end{align}
\normalsize
Finally, \cref{ID:nonsymmetric-quadraticForm} is used to combine the $\kronF{\bx}{p}$ factor on the left with the $\kronF{\bx}{k-p}$ on the right as desired:
\small
\begin{align}
     & = ij~\Vec \Biggl[ \left(\bI_{n^p} \otimes \Vec\left[ \bI_m\right]^\top\right)  \times \label{eq:G-terms}          \\
     & \left(\Vec\left[\G_q^\top \bV_j \right]^\top \otimes \left(\G_p^\top \bV_i \otimes \bI_m \right) \right) \bP^\top
        \left( \bI_{n^{k-p}} \otimes \Vec\left[ \bI_m\right] \right)\Biggr]^\top {\kronF{\bx}{k}}. \nonumber
\end{align}
\normalsize

We must also consider the case of \cref{eq:vi-G-terms-1} with $p=q=0$, which corresponds to terms of the form
$
    \frac{1}{8} i \bv_i^\top (\bI_n \otimes \kronF{\bx}{i-1}) \bB \bB^\top (\bI_n \otimes {\kronF{\bx}{j-1}}^\top) \bv_j j,
$
from \cref{eq:Hinfty-Past-HJB-plugged-in-symmetric-2,eq:Hinfty-Past-HJB-plugged-in-symmetric-3}.
The degree~$k$ terms of this form occur when $i+j -2= k$, excluding the cases $i=2$ or $j=2$ since those terms contain a factor of $\bv_k$ and thus were accounted for in \cref{eq:kth-terms-sofar-2}.
Following the same approach used to simplify \cref{eq:vi-G-terms-1}, these terms can be written as
\begin{align}
     & \frac{1}{8} ij~\Vec \left[ \bV_i^\top \bB \bB^\top \bV_j \right]^\top {\kronF{\bx}{k}}. \label{eq:B-terms}
\end{align}

Lastly, using \cref{ID:nonsymmetric-quadraticForm}, an arbitrary $k$th-order term from \cref{eq:Hinfty-Past-HJB-plugged-in-symmetric-4} containing $\Hx_p$ can be written as
\begin{align}
    -\frac{\eta}{2} {\kronF{\bx}{p}}^\top \Hx_p^\top \Hx_q \kronF{\bx}{q} & = -\frac{\eta}{2} \Vec[\Hx_p^\top \Hx_q]^\top \kronF{\bx}{k} \label{eq:H-terms}
\end{align}
with $p+q=k$.
The terms \cref{eq:N-terms,eq:G-terms,eq:B-terms,eq:H-terms} represent single terms of degree~$k$ containing the contributions of individual $\bF_p$, $\bG_p$, and $\bH_p$ coefficients from the dynamics.
To collect all of the terms of degree~$k$ and explicitly write what was labelled as ``(everything else)'' in \cref{eq:kth-terms-sofar-2}, we need to introduce summations over the indices $p$ \& $i$ from \cref{eq:N-terms}, $p$,$q$,$i$,$j$, \& $o$ from \cref{eq:G-terms}, and $p$ \& $q$ from \cref{eq:H-terms}.
The collection of $k$th-order terms in the HJB equation \cref{eq:Hinfty-Past-HJB-plugged-in-symmetric} for $3 \leq k \leq d$ can then be written as
\small
\begin{align}
     & 0 = \frac{1}{2} \bv_k^\top \cL_k(\bA +  \bB \bB^\top \bV_2 ) \kronF{\bx}{k}
    +  \frac{1}{2} \sum_{\mathclap{\substack{i,p\geq 2                                                                                           \\       i + p = k+1}}} \bv_i^\top \cL_i(\bF_p) \kronF{\bx}{k} \\
     & + \frac{1}{8} \sum_{\substack{i,j>2                                                                                                       \\ i+j=k+2}} \!\!\!\!ij~\Vec[\bV_i^\top \bB \bB^\top \bV_j]^\top \kronF{\bx}{k} \nonumber\\
     & + \frac{1}{8} \sum_{o=1}^{2\ell}\left( \sum_{\substack{p,q \geq 0                                                                         \\p+q=o}} \left(  \sum_{\substack{i,j\geq 2 \\       i+j=k-o+2}} \!\!\!\!ij~\Vec \Biggl[ \left(\bI_{n^p} \otimes \Vec\left[ \bI_m\right]^\top\right) \times  \right.\right. \nonumber\\
     & \hspace{0.25cm} \left(\Vec\left[\G_q^\top \bV_j \right]^\top \otimes \left(\G_p^\top \bV_i \otimes \bI_m \right) \right) \times \nonumber \\
     & \hspace{0.25cm} \left.\left.\bP^\top
        \left( \bI_{n^{k-p}} \otimes \Vec\left[ \bI_m\right] \right)\Biggr]^\top
    \vphantom{\sum_{\substack{p,q \geq 0                                                                                                         \\p+q=o}}} \right)\right){\kronF{\bx}{k}} -\frac{\eta}{2} \sum_{\mathclap{\substack{p,q\geq 1 \\ p + q = k}}} \Vec[\Hx_p^\top \Hx_q]^\top \kronF{\bx}{k}. \nonumber
\end{align}
\normalsize

Requiring this to hold for all $\bx$, we pull out
the factor of ${\kronF{\bx}{k}}$ from every term, multiply by two, and transpose the entire equation to obtain the linear system \cref{eq:LinSysForVk-Poly} to solve for the unknown coefficient $\mathbf{\tv}_k$. \hspace*{\fill}~\QED\par\endtrivlist\unskip

The proof for \cref{thm:wiPoly} mirrors the proof for \cref{thm:viPoly}; the only difference is that some of the coefficients in the HJB PDE \cref{eq:Hinfty-Future-HJB} are interchanged relative to those in \cref{eq:Hinfty-Past-HJB}.

\begin{remark}
    Equation \cref{eq:kth-terms-sofar-2} has a unique \textit{symmetric} solution, but there are many non-symmetric $\mathbf{\tv}_k$ that satisfy the equation as well.
    To give an analogy, observe that $\bx^\top \bP \bx = \bx^\top \bQ \bx$ for all $\bx$ does not imply $\bP = \bQ$; however, the symmetrizations of $\bP$ and $\bQ$ are equivalent.
    Thus, once the solution $\mathbf{\tv}_k$ is computed, we impose a symmetrization step to ensure that $\bv_k$ is the unique symmetric solution to \cref{eq:kth-terms-sofar-2}.
\end{remark}
\begin{remark}
    In \cite{Kramer2024}, the only nonlinear term included in the analysis is $\bF_2$, which eliminates everything from \cref{eq:vi-G-terms-1} to \cref{eq:H-terms} except for \cref{eq:B-terms}.
    The contribution of the present work is the ability to include all of the $\bF_p$, $\bG_p$, and $\bH_p$.
    The symmetry provided by \cref{thm:Past-energy-deriv-compact} is the key insight which enables compactly expressing all of the combinations of terms to be included in the right-hand-side of \cref{eq:LinSysForVk-Poly}.
\end{remark}

\begin{theorem}
    Let $\gamma > \gamma_0 \geq 0$ and $\eta = 1-\gamma^{-2}$, as in \cref{thm:viPoly,thm:wiPoly}.
    Then the equations \cref{eq:LinSysForVk-Poly,eq:LinSysForWk-Poly} have unique solutions.
\end{theorem}
\begin{proof}
    Under the assumptions of \cref{thm:viPoly,thm:wiPoly}, which are the assumptions of \cite[Thm. 5.2]{Scherpen1996},
    $\left(\bA + \bB \bB^\top  \bV_2 \right)$ and $\left(\bA - \eta \bB \bB^\top  \bW_2\right)$ are asymptotically stable, and hence nonsingular.
    This implies that $\cL_{k} \left(\bA + \bB \bB^\top  \bV_2 \right)^\top$ and $\cL_{k} \left(\bA - \eta \bB \bB^\top  \bW_2 \right)^\top$ are invertible \cite{Chen2019},
    and the linear systems \cref{eq:LinSysForVk-Poly,eq:LinSysForWk-Poly} have unique solutions.
\end{proof}

\cref{alg:alg1} summarizes the process for computing energy function approximations using \cref{thm:viPoly,thm:wiPoly}.
\begin{algorithm}[h!]
    \caption{Computing Taylor approximations to the $\cH_\infty$ balancing past and future energy functions $\cE_\gamma^-(\bx)$ and $\cE_\gamma^+(\bx)$}\label{alg:alg1}
    \begin{algorithmic}[1]
        \Require System matrices $\bA$, $\bB$, $\bC$, $\{\bF_p\}_{p=2}^\ell$, $\{\bG_p\}_{p=2}^\ell$, $\{\bH_p\}_{p=2}^\ell$; desired approximation degree $d$; $\cH_\infty$ gain parameter $\eta \coloneqq 1-\gamma^{-2}$.
        \Ensure Polynomial coefficients $\{\bv_i\}_{i=2}^d$ and $\{\bw_i\}_{i=2}^d$
        \State Solve the AREs \cref{eq:V2-2,eq:W2-2} for $\bv_2 = \Vec\left[\bV_2\right]$, $\bw_2 = \Vec\left[\bW_2\right]$.
        \For{$k=3$ to $d$}
        \State Form the linear systems \cref{eq:LinSysForVk-Poly,eq:LinSysForWk-Poly};
        \State solve \cref{eq:LinSysForVk-Poly,eq:LinSysForWk-Poly} for $\tilde\bv_k$ and $\tilde\bw_k$;
        \State symmetrize $\tilde\bv_k$ and $\tilde\bw_k$ to obtain $\bv_k$ and $\bw_k$.
        \EndFor
    \end{algorithmic}
\end{algorithm}
\subsection{Computational Complexity Analysis}\label{sec:complexity}
Using the floating point operation (flop) counts for standard BLAS operations, we evaluate the computational complexity of forming and solving the linear system \cref{eq:LinSysForVk-Poly}.
\begin{proposition}
    Consider a nonlinear dynamical system with degree $\ell$ polynomial structure \cref{eq:FOM-Poly} with state dimension $n$, $m$ inputs, $p$ outputs.
    The cost of computing degree~$d$ approximations to the past and future energy functions with \cref{thm:viPoly,thm:wiPoly} is $O(dn^{d+1})$.
\end{proposition}
\begin{proof}
    Here we show the cost of forming and solving the linear system \cref{eq:LinSysForVk-Poly}; the cost for the linear system \cref{eq:LinSysForWk-Poly} is identical.
    First, we consider the flops required to form the terms on the right-hand side of the linear system \cref{eq:LinSysForVk-Poly}.
    Consider the first set of terms in \cref{eq:LinSysForVk-Poly},
    \begin{align}\label{eq:Fterms}
        -\sum_{\mathclap{\substack{i,p\geq 2 \\ i + p = k+1}}} \cL_i(\bF_p)^\top \bv_i.
    \end{align}
    The matrix $\cL_i(\bF_p)$ has dimension $(n^i \times n^k)$,
    whereas vector $\bv_i$ is $(n^i \times 1)$,
    so the cost of evaluating the Lyapunov product $\cL_i(\bF_p)^\top \bv_i$ using naive matrix-vector multiplication is $O(n^{k+i})$ using level-2 BLAS operations.
    The dominant cost in the sum is therefore the term with $i=k-1$, for a total cost of $O(n^{2k-1})$.
    Instead, we exploit the structure of the \emph{i-way Lyapunov matrix} to form these terms more efficiently.
    A term from \cref{eq:Fterms} can be expanded as
    \begin{align*}
        \cL_i(\bF_p)^\top \bv_i = (\bF_p^\top \otimes \bI_{n^{i-1}}) \bv_i + (\bI_n \otimes \bF_p^\top \otimes \bI_{n^{i-2}}) \bv_i + \dots
    \end{align*}
    All of the neglected terms are computed similarly to the first term with an appropriate permutation/reshaping, so the total cost is $i$ times the cost of computing the first term.
    Using \cref{ID:T3.4}, we rewrite the first term in the sum as
    $(\bF_p^\top \otimes \bI_{n^{i-1}}) \bv_i  = \Vec[\bV_i^\top \bF_p]$,
    which is now matrix multiplication of $(n^{i-1} \times n)$ and $(n \times n^p)$ matrices, which has a cost of $O(n^{i+p})$ using level-3 BLAS operations.
    Since $i+p = k+1$, this is equivalent to $O(n^{k+1})$.
    Performing this operation $i$ times for the remaining terms, the total cost of evaluating $\cL_i(\bF_p)^\top \bv_i$ this way is $O(i n^{k+1})$.
    The dominant cost occurs for the case $i = k - 1$, so the total cost to form the set of terms \cref{eq:Fterms} is $O(k n^{k+1})$.

    Next consider the terms
    \begin{align}\label{eq:Bterms}
        -\frac{1}{4}\sum_{\mathclap{\substack{i,j>2 \\ i+j=k+2}}} ij~\Vec[\bV_i^\top \bB \bB^\top \bV_j].
    \end{align}
    Since products like $\bV_i^\top \bB$ appear repeatedly, we can store them in memory to avoid repeatedly forming them; however, the dominant cost comes from multiplying these stored quantities together.
    Treating $\bV_i^\top \bB$ as an $(n^{i-1} \times m)$ matrix and $\bB^\top \bV_j$ as an $(m \times n^{j-1})$ matrix, the multiplication $\bV_i^\top \bB \bB^\top \bV_j$ costs $O(m n^{i+j-2})$ using level-3 BLAS operations.
    Since $i+j = k+2$, this is equal to $O(m n^k)$.
    We form $k-2$ of these terms in the sum, so the overall cost is $O(k m n^k)$.

    Next consider the sum
    \begin{align}
        \eta\sum_{\mathclap{\substack{p,q\geq 1 \\ p + q = k}}} \Vec[\Hx_p^\top \Hx_q].
    \end{align}
    The matrix product $\Hx_p^\top \Hx_q$ costs $O(n^k)$ using level-3 BLAS operations, and we form $k-1$ of these terms in the sum for a total cost of $O(kn^k)$.

    Finally, consider the sums
    \begin{equation}\label{eq:Gterms}
        \begin{split}
             & - \frac{1}{4} \sum_{o=1}^{2\ell}\left( \sum_{\substack{p,q \geq 0                                                             \\p+q=o}} \left(  \sum_{\substack{i,j\geq 2 \\       i+j=k-o+2}} \!\!\!\!\!\!ij~\Vec \Biggl[ \left(\bI_{n^p} \otimes \Vec\left[ \bI_m\right]^\top\right) \times \right.\right.\\
             & \hspace{.4cm} \left(\Vec\left[\G_q^\top \bV_j \right]^\top \otimes \left(\G_p^\top \bV_i \otimes \bI_m \right) \right) \times \\
             & \hspace{.4cm} \left.\left. \left(\bI_{n^{j-1}} \otimes \perm{n^{i-1}}{n^q m} \otimes \bI_m \right)
                    \left( \bI_{n^{k-p}} \otimes \Vec\left[ \bI_m\right] \right)\Biggr]
            \vphantom{\sum_{\substack{p,q \geq 0                                                                                             \\p+q=o}}} \right)\right).
        \end{split}
    \end{equation}

    Here the situation is more subtle, as the products $(\bI_{n^p} \otimes \Vec\left[\bI_m\right]^\top)$, $\left(\bI_{n^{j-1}} \otimes \perm{n^{i-1}}{n^q m} \otimes \bI_m \right)$, and $(\bI_{n^{k-p}} \otimes \Vec\left[\bI_m\right])$ are operations performed on sparse binary matrices consisting only of integer 1s and 0s; therefore, no floating point operations are performed.
    These are primarily memory operations as opposed to flops.
    Furthermore, many programming languages, including \textsc{Matlab}, form Kronecker products of sparse binary matrices very efficiently.
    The dominant cost in these terms is therefore from operations involving $\G_q^\top \bV_j$.
    Since products like $\G_q^\top \bV_j$ are used repeatedly, we store them in memory to avoid repeatedly forming them.
    The vector $\Vec\left[\G_q^\top \bV_j \right]^\top$ has dimension
    $(1 \times m n^{q+j-1})$.
    The matrix $\left(\G_p^\top \bV_i \otimes \bI_m \right)$ has dimension
    $(m^2n^p \times  mn^{i-1})$.
    So the Kronecker product $\Vec\left[\G_q^\top \bV_j \right]^\top \otimes \left(\G_p^\top \bV_i \otimes \bI_m \right)$
    costs $O(m^4 n^{p+q+i+j-2})$, which is $O(m^4 n^{k})$.

    We next consider the cost of summation of all terms in \cref{eq:Gterms}.
    There are $(k-o-1)$ terms in each innermost sum, which is at most $k$ terms.
    The middle sum is $(o+1)$ terms, which is at most $2\ell+1$ terms.
    The outermost sum is $2\ell$ terms.
    So there are in total $O(2k \ell (2\ell+1))$ terms to form,
    and the total cost is $O(k \ell^2 m^4 n^{k})$.

    To summarize, the \emph{total cost of forming} the right-hand sides of the linear systems \cref{eq:LinSysForVk-Poly,eq:LinSysForWk-Poly} is $O(kn^{k+1}) + O(kmn^k) + O(kn^k) + O(k \ell^2 m^4 n^{k})$; assuming $n \gg km$ and $n \gg k \ell^2 m^4$, the cost of forming the right-hands sides is $O(kn^{k+1})$.

    To \emph{solve} the linear systems, a naive approach would require solving a linear system of dimension $n^k$ for the $k$th coefficient, which has a cost of $O(n^{3k})$ using a direct method.
    Instead, we leverage the \emph{k-way Lyapunov matrix structure} of the left-hand sides and use the efficient solver introduced in \cite{Chen2019}, which has a computational complexity of $O(kn^{k+1})$.
    Overall then, the cost of \emph{forming and solving} the linear systems \cref{eq:LinSysForVk-Poly,eq:LinSysForWk-Poly} for the coefficient $\bv_k$ is $O(kn^{k+1})$, as opposed to a naive approach which costs $O(n^{3k})$.
    Since the highest order coefficient $\bv_d$ is the most expensive to compute, the overall cost is $O(dn^{d+1})$.
\end{proof}

\section{Illustrative Examples}\label{sec:results}
In \cref{sec:example1}, we examine a simple 1D example for which we know the true solutions to the energy functions.
This allows us to compare the accuracy of the computed energy functions to the true solutions.
In \cref{sec:example2}, we consider a slightly more complex 2D problem for which the energy functions can be visualized as contour plots.
The \texttt{cnick1/NLbalancing} repository \cite{NLBalancing2023} provides the functions \texttt{approxPastEnergy()} and \texttt{approxFutureEnergy()} for computing energy functions using \cref{thm:viPoly,thm:wiPoly};
the script \texttt{examplesForPaper3} reproduces all of the results in this paper.

\subsection{1D Example}\label{sec:example1}
Consider the 1D polynomial model
\begin{align*}
    \dot{x} & = ax + nx^2 + bu +g_1xu + g_2x^2u, &
    y       & = cx,
\end{align*}
where $a$, $b$, $c$, $n$, $g_1$, and $g_2$ are scalars, as are the state $x(t)$, input $u(t)$, and output $y(t)$.
Example 1 in \cite{Kramer2024} is a simplified case of this model with $g_1$ and $g_2$ set to zero;
here they are nonzero, resulting in a polynomial input map that cannot be handled by \cite{Kramer2024}.
The past and future energy functions are computed analytically and then compared with our approximations of increasing polynomial degree.

Since this model has only a single state dimension, the HJB PDEs reduce to 1D algebraic quadratic equations for the derivatives of the energy functions $\rd \cE_\gamma^{-}(x)/\rd x$ and $\rd \cE_\gamma^{+}(x)/\rd x$.
It is therefore straightforward to obtain the true energy functions $\cE_\gamma^{-}(x)$ and $\cE_\gamma^{+}(x)$ via traditional numerical integration.
The true past and future energy function solutions $\cE_{\gamma,\text{true}}^{-}(x)$ and $\cE_{\gamma,\text{true}}^{+}(x)$ are plotted in \cref{fig:example1b,fig:example1a}, respectively, for $a = -2$, $n = 1$, $b = 2$, $g_1 = -0.2$, $g_2 = 0.2$, $c = 2$, and $\eta = 0.5$.
In addition to the ground-truth energy functions, we plot degree~2, 4, 6, and 8 approximations.
\begin{figure}[htb]
    \centering
    \begin{subfigure}[]{\columnwidth}
        \begin{tikzpicture}
            \begin{axis}[
                    width=\textwidth, height = 5cm,
                    xmin=-6, xmax=6,
                    ymin=-1.5, ymax=10,
                    xlabel=$x$,
                    ylabel=$\mathcal{E}_\gamma^-(x)$,
                    legend style={at={(0.675,1)},draw=none,font=\scriptsize\sffamily,cells={align=center}},
                    xlabel style={font=\small\sffamily},
                    ylabel style={font=\small\sffamily},
                ]
                \addplot
                [color=blue,mark=+,each nth point={5}, mark size = 2pt, only marks]
                table [x index=0, y index=1, col sep=space] {ex1_past_a.txt};
                \addlegendentry{ground-truth}
                \addplot
                [color=Set1-A, thick, dashed, dash pattern=on 4pt off 1pt on 1pt off 1pt]
                table [x index=0, y index=1, col sep=space] {ex1_past_2.txt};
                \addlegendentry{degree 2}
                \addplot
                [color=Set1-E, thick, dashed, dash pattern=on 4pt off 1pt on 1pt off 1pt on 1pt off 1pt]
                table [x index=0, y index=1, col sep=space] {ex1_past_4.txt};
                \addlegendentry{degree 4}
                \addplot
                [color=Set1-D, thick, dashed, dash pattern=on 4pt off 1pt on 1pt off 1pt on 1pt off 1pt on 1pt off 1pt]
                table [x index=0, y index=1, col sep=space] {ex1_past_6.txt};
                \addlegendentry{degree 6}
                \addplot
                [color=Set1-C, thick, dashed, dash pattern=on 4pt off 1pt on 1pt off 1pt on 1pt off 1pt on 1pt off 1pt on 1pt off 1pt]
                table [x index=0, y index=1, col sep=space] {ex1_past_8.txt};
                \addlegendentry{degree 8}
            \end{axis}
        \end{tikzpicture}
        \caption{Past energy function and its approximations.}
        \label{fig:example1b}
    \end{subfigure}

    \begin{subfigure}[]{\columnwidth}
        \begin{tikzpicture}
            \begin{axis}[
                    width=\textwidth, height = 5cm,
                    xmin=-6, xmax=6,
                    ymin=0, ymax=20,
                    xlabel=$x$,
                    ylabel=$\mathcal{E}_\gamma^+(x)$,
                    clip mode=individual, 
                    legend style={at={(0.675,1)},draw=none,font=\scriptsize\sffamily,cells={align=center}},
                    xlabel style={font=\small\sffamily},
                    ylabel style={font=\small\sffamily},
                ]
                \addplot
                [color=blue,mark=+,each nth point={5}, mark size = 2pt, only marks]
                table [x index=0, y index=1, col sep=space] {ex1_future_a.txt};
                \addlegendentry{ground-truth}
                \addplot
                [color=Set1-A, thick, dashed, dash pattern=on 4pt off 1pt on 1pt off 1pt]
                table [x index=0, y index=1, col sep=space] {ex1_future_2.txt};
                \addlegendentry{degree 2}
                \addplot
                [color=Set1-E, thick, dashed, dash pattern=on 4pt off 1pt on 1pt off 1pt on 1pt off 1pt]
                table [x index=0, y index=1, col sep=space] {ex1_future_4.txt};
                \addlegendentry{degree 4}
                \addplot
                [color=Set1-D, thick, dashed, dash pattern=on 4pt off 1pt on 1pt off 1pt on 1pt off 1pt on 1pt off 1pt]
                table [x index=0, y index=1, col sep=space] {ex1_future_6.txt};
                \addlegendentry{degree 6}
                \addplot
                [color=Set1-C, thick, dashed, dash pattern=on 4pt off 1pt on 1pt off 1pt on 1pt off 1pt on 1pt off 1pt on 1pt off 1pt]
                table [x index=0, y index=1, col sep=space] {ex1_future_8.txt};
                \addlegendentry{degree 8}
            \end{axis}
        \end{tikzpicture}
        \caption{Future energy function and its approximations.}
        \label{fig:example1a}
    \end{subfigure}
    \caption{Past and future energy function approximations with $\eta = 1/2$ for Example 1.
        Higher-order polynomials are needed to approximate the non-quadratic energy functions inherent to the nonlinear system.}
    \label{fig:example1}
\end{figure}
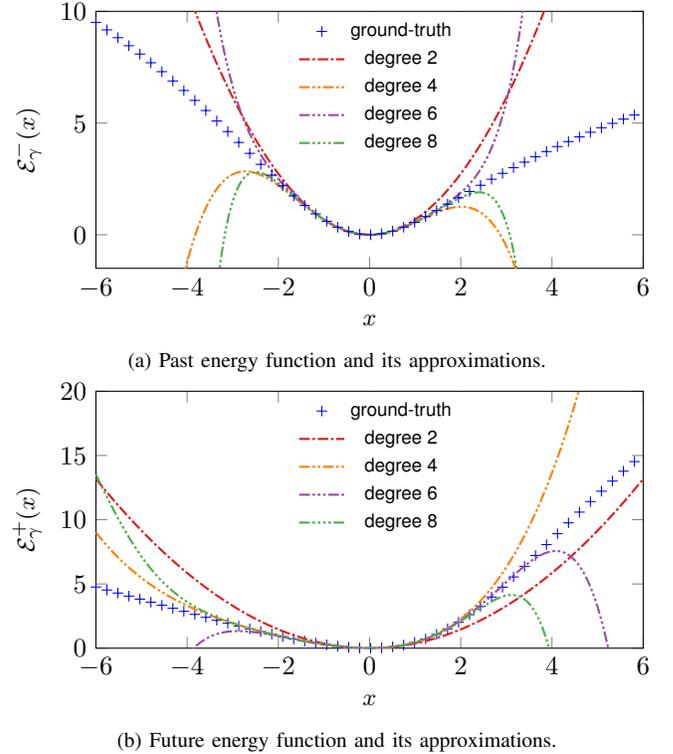

This example demonstrates that for even simple nonlinear systems, the energy functions are in general not quadratic, as is evident in \cref{fig:example1}.
This suggests that model reduction methods such as those based on algebraic Gramians \cite{Condon2005,Benner2024}, which compute quadratic approximations to the nonlinear energy functions, may fail to capture important features.

As is often the case with polynomial approximation, higher-degree approximations also tend to diverge more quickly beyond some local region of convergence.
To quantify the accuracy of the various degree approximations and their regions of convergence,
we introduce the remainder for the degree $d$ approximation $\cE_{\gamma}^+(x)$ to the true energy function $\cE_{\gamma,\text{true}}^+(x)$ as
\begin{align*}
    R^+_d(x) =  \cE_{\gamma,\text{true}}^+(x) - \cE_{\gamma}^+(x).
\end{align*}
The $L^\infty$-norm of the remainder $R^+_d(x)$ over the interval $-\chi$ to $\chi$ serves as an error metric.
In \cref{sfig:example1_regionOfAccuracy}, we vary $\chi$ from $0$~to~$6$ and select 250 evaluation points in this interval to show the errors for the various approximations to the past energy function $\cE_\gamma^+(x)$ ($\eta = 1/2$) for Example~1.
According to Taylor's theorem, there is a neighborhood within which the remainder tends to zero $R^+_d(x) \to 0$ as we continue to add higher-order terms to the approximation $d \to \infty$ \cite{Marsden1998a}; however, Taylor's theorem does not specify the size of this neighborhood.
\cref{sfig:example1_regionOfAccuracy} clearly demonstrates the implication of Taylor's theorem:
higher-order approximations \emph{are} more accurate locally, yet the region of convergence is not widened despite the additional terms in the approximations.
\pgfdeclarelayer{background}
\pgfsetlayers{background,main}
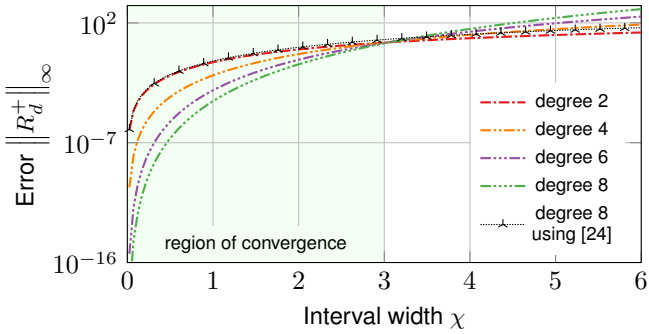
\begin{figure}[htb]
    \centering
    \begin{tikzpicture}[]
        \begin{semilogyaxis}[
                width = 0.95\columnwidth, height = 5cm,
                grid,
                xmin=0, xmax=6,
                ymin=1e-16, ymax=2e3,
                xlabel={Interval width $\chi$},
                ylabel={Error $\left\lVert R^+_d\right\rVert_\infty$},
                clip mode=individual, 
                legend pos=south east,
                legend style={draw=none,cells={align=center},font=\scriptsize\sffamily},
                xlabel style={font=\small\sffamily},
                ylabel style={font=\small\sffamily},
                ylabel shift = -6 pt,
            ]
            \begin{pgfonlayer}{background}
                \fill[color=green!5] (axis cs:0,1e-16) rectangle (axis cs:3,2e3);
            \end{pgfonlayer}
            \addplot
            [color=Set1-A, thick, dashed, dash pattern=on 4pt off 1pt on 1pt off 1pt]
            table [x index=0, y index=1, col sep=&] {example1_regionOfAccuracy.dat};
            \addlegendentry{degree 2}
            \addplot
            [color=Set1-E, thick, dashed, dash pattern=on 4pt off 1pt on 1pt off 1pt on 1pt off 1pt]
            table [x index=0, y index=2, col sep=&] {example1_regionOfAccuracy.dat};
            \addlegendentry{degree 4}
            \addplot
            [color=Set1-D, thick, dashed, dash pattern=on 4pt off 1pt on 1pt off 1pt on 1pt off 1pt on 1pt off 1pt]
            table [x index=0, y index=3, col sep=&] {example1_regionOfAccuracy.dat};
            \addlegendentry{degree 6}
            \addplot
            [color=Set1-C, thick, dashed, dash pattern=on 4pt off 1pt on 1pt off 1pt on 1pt off 1pt on 1pt off 1pt on 1pt off 1pt]
            table [x index=0, y index=4, col sep=&] {example1_regionOfAccuracy.dat};
            \addlegendentry{degree 8}
            \addplot
            [color=black, thin, mark=Mercedes star,mark repeat=6, mark options=solid, dashed, dash pattern=on .5pt off .5pt on .5pt off .5pt]
            table [x index=0, y index=12, col sep=&] {example1_regionOfAccuracy.dat};
            \addlegendentry{degree 8 \\ using \cite{Kramer2024}}

            \node at (axis cs:1.5,1e-16) [above, fill=green!5] {\scriptsize\sffamily region of convergence};

        \end{semilogyaxis}
    \end{tikzpicture}

    \caption{
        Errors for the various approximations to the future energy function $\cE_\gamma^+(x)$ for Example~1 ($\eta = 1/2$) on intervals $(-\chi,\chi)$.
        Convergence with increasing polynomial degree occurs in a neighborhood of the origin,
        as predicted by Taylor's theorem.
        The degree~4 approximation including all of the terms in the dynamics is superior to the degree~8 approximation computed using \cite{Kramer2024}, which neglects terms in the dynamics.
    }
    \label{sfig:example1_regionOfAccuracy}
\end{figure}

\cref{sfig:example1_regionOfAccuracy} also contains the error for a degree 8 approximation computed with the algorithm in \cite{Kramer2024}.
Since the method therein assumes quadratic drift, linear inputs, and linear outputs, this approximation amounts to neglecting $g_1$ and $g_2$.
Interestingly, the degree~4 approximation including the full system dynamics is superior to the degree~8 approximations with dynamics neglecting $g_1$ and $g_2$;
in fact, the degree~8 approximation neglecting terms from the dynamics does not appear to be much more accurate than the quadratic approximation coming from linearizing the system.
These results indicate that, regarding local accuracy, including all of the information from the dynamics is more important than computing a higher-order energy function approximation.
In other words, locally, a lower-order approximation to the \emph{correct} energy function is better than a higher-order approximation to the \emph{wrong} energy function.

\subsection{2D Example}\label{sec:example2}
Consider the 2D quadratic-bilinear system from \cite{Kawano2017}:
\begin{align*}
    \dot{x}_1 & = -x_1 + x_2 - x_2^2 + u + 2 x_2 u, &
    \dot{x}_2 & = -x_2 + u,                           \\
    y         & = x_1.
\end{align*}

Since the state is two-dimensional, the computed energy functions can be visualized as contour plots;
\cref{sfig:example2_past,sfig:example2_future} show the open-loop controllability and observability energy functions computed with our method by setting $\eta=0$.
Quadratic functions in two dimensions have elliptical contours,
so the curvature and asymmetry present in these energy functions indicates that they are highly non-quadratic.
Similar to the 1D example, we observe that an algebraic Gramian-based approach would necessarily result in quadratic energy functions, so such approaches fail to capture the true behavior of these energy functions.
\begin{figure}[htb]
    \centering
    \begin{subfigure}{0.2375\textwidth}
        \includegraphics[height=5cm]{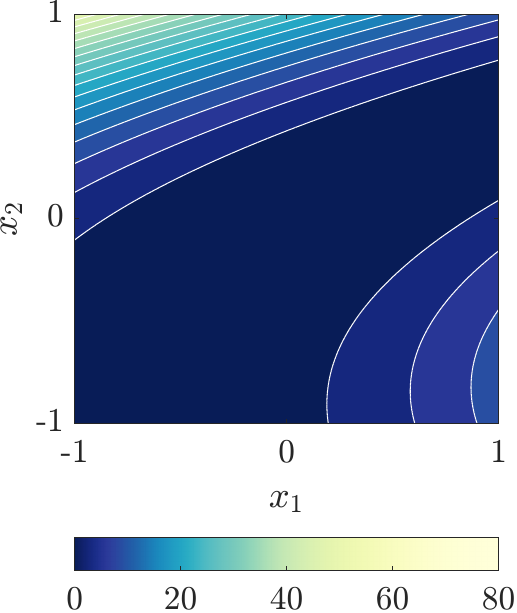}
        \caption{Past energy $\cE_\gamma^{-}(\bx)$.}
        \label{sfig:example2_past}
    \end{subfigure}
    \hfill
    \begin{subfigure}{0.2375\textwidth}
        \includegraphics[height=5cm]{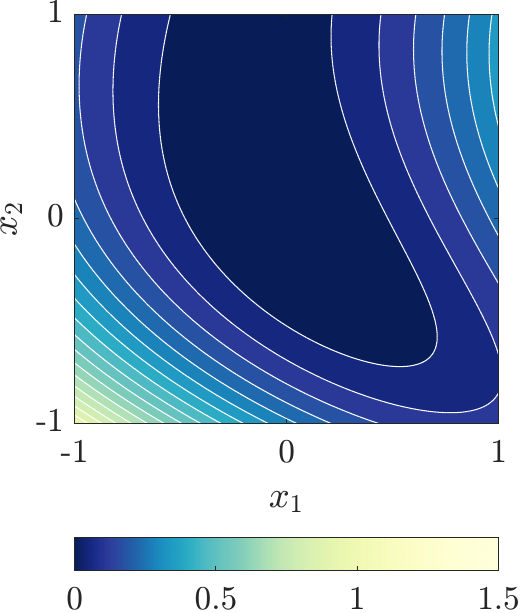}
        \caption{Future energy $\cE_\gamma^{+}(\bx)$.}
        \label{sfig:example2_future}
    \end{subfigure}

    \begin{subfigure}{0.2375\textwidth}
        \includegraphics[height=5cm]{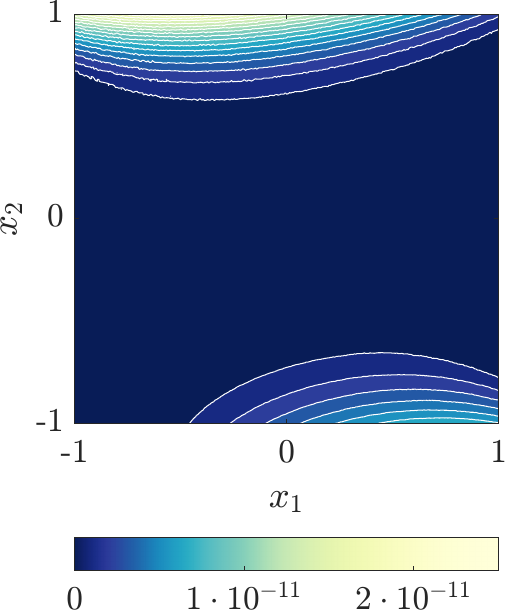}
        \caption{Past HJB residual $\text{RES}^-(\bx)$.}
        \label{sfig:example2_past_res}
    \end{subfigure}
    \hfill
    \begin{subfigure}{0.2375\textwidth}
        \includegraphics[height=5cm]{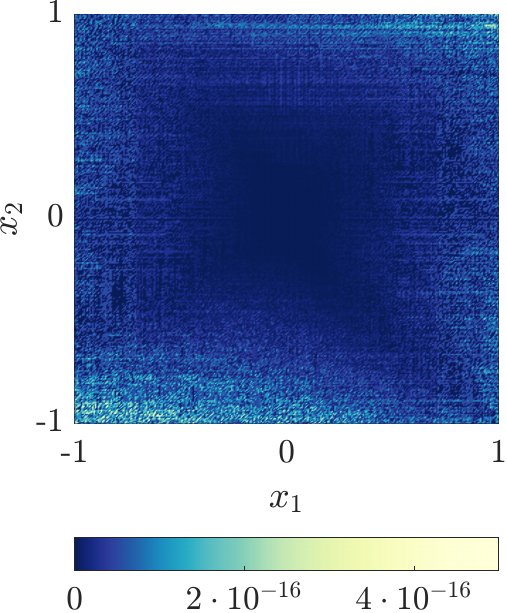}
        \caption{Future HJB residual $\text{RES}^+(\bx)$.}
        \label{sfig:example2_future_res}
    \end{subfigure}
    \caption{(a) Past energy function $\cE_\gamma^{-}(\bx)$ degree 4 approximation with $\eta=0$, (b) future energy function $\cE_\gamma^{+}(\bx)$ degree 4 approximation with $\eta=0$, (c) past energy function HJB residual $\text{RES}^-(\bx)$, (d) future energy function HJB residual $\text{RES}^+(\bx)$.}
    \label{fig:example2_energyFunctions}
\end{figure}

Since the true analytical energy functions are not available to compare with in this example, we use the HJB residual to assess the accuracy of the approximations in \cref{fig:example2_energyFunctions}.
Based on \cref{eq:Hinfty-Past-HJB}, define the residual for the past HJB PDE as
\begin{equation*} 
    \begin{split}
        \text{RES}^{-}(\bx) & =  \Bigg|\frac{\partial \cE_\gamma^{-}(\bx)}{\partial \bx} \bf(\bx) + \frac{1}{2}  \frac{\partial \cE_\gamma^{-}(\bx)}{\partial \bx} \bg(\bx) \bg(\bx)^\top \frac{\partial^\top \cE_\gamma^{-}(\bx)}{\partial \bx} \\
                            & \qquad - \frac{\eta}{2}  \bh(\bx)^\top  \bh(\bx)\Bigg|.
    \end{split}
\end{equation*}
The residual for the future HJB PDE is defined similarly based on \cref{eq:Hinfty-Future-HJB}.
The HJB residual has been used as an error metric in optimization-based approaches in the literature, see e.g. \cite{Borovykh2022}.
The HJB residuals corresponding to the energy functions shown in \cref{sfig:example2_past,sfig:example2_future} are shown in \cref{sfig:example2_past_res,sfig:example2_future_res}.
Note that the residual is zero when the HJB PDE is satisfied.

For this
model, the degree~4 solutions are sufficient to
accurately approximate the energy functions on the domain from -1 to 1.
The future HJB PDE is satisfied nearly to machine precision.
The past HJB residual is very small throughout most the domain of interest, and it only grows
at the edges of the domain where the energy function in \cref{sfig:example2_past} also grows steeply.
These regions appear to be more difficult for the degree~4 polynomial to approximate.
Nonetheless, the energy function values are on the order of $O(10)$ whereas the HJB residual errors are on the order of $O(10^{-11})$.

The previous results serve to illustrate the expected behavior of solutions computed our proposed method.
Ultimately, the method computes a Taylor approximation to the energy functions, so the classical results from Taylor's theorem apply: the solutions have guaranteed convergence in a neighborhood of the origin, but in general care must be taken to check that the region of interest is included in the region of convergence.
Outside of the region of convergence, polynomial approximations quickly diverge to $\pm \infty$.

\section{Numerical Results for a New High-Dimensional Benchmark Problem}\label{sec:fem}

In this section, we seek to demonstrate the
scalability and convergence of the proposed algorithms
on an Euler-Bernoulli cantilever beam finite element model with von Kármán geometric nonlinearity (see \cref{fig:beam}), based on an example from Reddy~\cite{Reddy2004}.
We emphasize that while Al'brekht's method has been used often in the literature, the contribution in the present work lies in the ability to apply the method to significantly higher dimensional systems, such as those which may require model reduction via nonlinear BT.
The results are obtained on a
Linux workstation with an Intel Xeon W-3175X CPU, 256 GB RAM, and \textsc{Matlab} 2021a.

\begin{figure}[htb]
    \centering
    \includegraphics[width=\columnwidth]{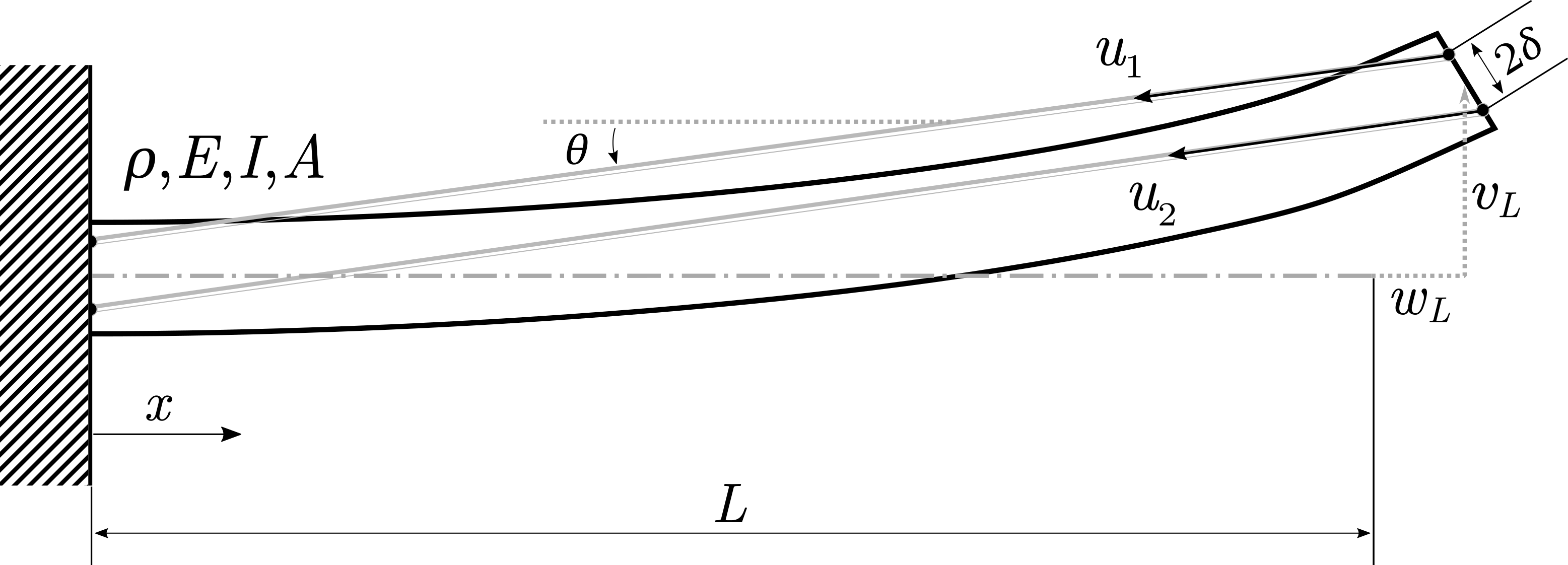}
    \caption{A cable-actuated cantilever beam.
        The cables are the solid light gray lines with a spacing of $2\delta$, through-which the control inputs $u_1$ and $u_2$ are applied.
        When deformed, the beam deviates from the dot-dashed centerline by $w(x,t)$ and $v(x,t)$ in the horizontal and vertical directions, respectively.
        The displacement of the tip, given by $w_L \coloneqq w(L,t)$ and $v_L \coloneqq v(L,t)$, determines the cable angle, $\theta$, as the beam deforms.}
    \label{fig:beam}
\end{figure}

\subsection{Model Derivation}
The basic example from Reddy~\cite{Reddy2004} is modified by adding cable actuators to construct an input-output control model of the form \cref{eq:FOM-Poly} with scalable state dimension $n$.
The cable actuation gives a state-dependent control-affine forcing term which can be approximated with arbitrarily high polynomial degree $\ell$.
The two tendon-like cables are attached a distance $\delta$ from the centerline (see \cref{fig:beam}).
When actuated together, the cables apply forcing in the horizontal direction, whereas when actuated independently, they produce a bending moment which can move the beam in the vertical direction.

The governing equations for the nonlinear Euler-Bernoulli beam are
\begin{align*}
    0 & = \rho A \frac{\partial^2 w}{\partial t^2}
    - \frac{\partial N_{xx}}{\partial x}    ,                                                                                        \\
    0 & = \rho A \frac{\partial^2 v}{\partial t^2} - \frac{\partial }{\partial x} \left(N_{xx} \frac{\partial v}{\partial x} \right)
    + \frac{\partial^2 M_{xx}}{\partial x^2},
\end{align*}
where
$v(x,t)$
and
$w(x,t)$
represent the beam's
transverse
and
longitudinal
deflections, respectively, as functions of position along the beam $x$ and time $t$.
The quantities $N_{xx}$ and $M_{xx}$ are the axial force and bending moment; we employ von Kármán geometric nonlinearity and write
\begin{align*}
    N_{xx} & = E A \left[\frac{\partial w}{\partial x}  + \frac{1}{2}\left(\frac{\partial v}{\partial x} \right)^2 \right] , &
    M_{xx} & = E I \frac{\partial^2 v}{\partial x^2}.
\end{align*}
The physical parameters are the
density $\rho$,
elastic modulus $E$,
second moment of area $I$,
and
cross-sectional area $A$.
The model is nonlinear due to the inclusion of the quadratic strain component in $N_{xx}$, without-which the linear Euler-Bernoulli beam and classical bar theories are recovered.

The boundary conditions for the fixed end of the beam are
\begin{align}
    v(0,t) & = 0, & w(0,t) & = 0, & \left. \frac{\partial w(x,t)}{\partial x} \right|_{x=0} & = 0.
\end{align}
The boundary conditions for the forced end of the beam correspond to the forces imparted by the cables, which enter though the secondary variables in the finite element formulation \cite{Reddy2004}.
Assuming a small cable attachment distance $\delta \ll 1$, these forces are
\begin{subequations}
    \begin{align}
        \left. N_{xx}\right|_{x=L}                                                                   & =-\left(u_1(t) + u_2(t)\right) \cos \theta ,      \label{eq:BC1} \\
        \left[\frac{\partial v}{\partial x} N_{xx} + \frac{\partial M_{xx}}{\partial x}\right]_{x=L} & =  -\left(u_1(t) + u_2(t)\right) \sin \theta ,    \label{eq:BC2} \\
        \left. M_{xx}\right|_{x=L}                                                                   & =\delta\left(u_1(t) - u_2(t)\right) \cos \theta \label{eq:BC3}.
    \end{align}
\end{subequations}
Note how the control inputs $\bu(t) = \begin{bmatrix}
        u_1(t) & u_2(t)
    \end{bmatrix}^\top$ enter through the boundary conditions.
The cable angle $\theta = \theta(\bx)$ is state dependent but can be expressed with simple geometry by
\begin{align*}
    \cos \theta & = \frac{L+w_L}{\sqrt{(L+w_L)^2+v_L^2}}, \quad
    \sin \theta  = \frac{v_L}{\sqrt{(L+w_L)^2+v_L^2}},
\end{align*}
where $w_L \coloneqq w(L,t)$ and $v_L \coloneqq v(L,t)$ represent the horizontal and vertical displacements, respectively, of the tip of the deformed beam, as shown in \cref{fig:beam}.
Approximating $\cos \theta$ and $\sin \theta$ to third-order via Taylor-series expansion with respect to $v_L$ and $w_L$ yields
\begin{align}
    \cos \theta & \approx \left(1 - \frac{v_L^2}{2 L^2} + \frac{w_L v_L^2}{L^3}\right) ,                                                      \label{eq:cosTheta} \\
    \sin \theta & \approx \left(\frac{v_L}{L} - \frac{w_L v_L}{L^2}  + \frac{\left(2 w_L^2 v_L-v_L^3\right)}{2 L^3}\right). \label{eq:sinTheta}
\end{align}

Since the cable forces enter through the secondary variables in the finite element formulation, the boundary conditions \cref{eq:BC1,eq:BC2,eq:BC3} represented with the polynomial expansions \cref{eq:cosTheta,eq:sinTheta} enter directly into the $\G_p$ matrices which define the polynomial structure of the input vector fields.
The work of Kramer et al.  \cite{Kramer2024} requires linear inputs corresponding to the zeroth order approximations  $\cos \theta  \approx 1$ and $\sin \theta \approx 0$.
This example adds three more orders to the input approximation, leading to a cubic semidiscretized system of the form \cref{eq:FOM-Poly} with $\ell=3$.

\subsection{Convergence and Scalability}
We investigate the convergence of the energy functions as the finite element mesh is refined.
The beam is prescribed initial conditions corresponding to a linear displacement field for the transverse and longitudinal directions:
\begin{align*}
    v(x,0) & = x_0 \frac{x}{L}, & w(x,0) & = x_0 \frac{x}{L} ,
\end{align*}
and we compute the future energy for an initial condition $\bx_a$ corresponding to $x_0 = 0.01$.
In \cref{tab:example6_convergenceData}, the number of elements in the finite element model is increased while keeping the degree of the energy function approximation fixed at $d=3$ and $d=4$, respectively.
Each additional element contributes 6 additional degrees of freedom.
The energy function values are shown in the second and third columns of \cref{tab:example6_convergenceData} for the initial condition $\bx_a$, and we see that as the mesh is refined, the energy function values do converge.
In \cref{fig:example6_convergence_n}, we plot the energy function values from the second and third columns of \cref{tab:example6_convergenceData} to more clearly show this convergence.

\pgfplotstableread[col sep=&]{example6_convergenceData_d3.dat}{\tableDataA}
\pgfplotstablegetelem{1}{exponentCoeff}\of\tableDataA
\edef\exponentCoeff{\pgfplotsretval}
\pgfplotstablegetelem{1}{exponentFit}\of\tableDataA
\edef\exponentFit{\pgfplotsretval}

\pgfplotstableread[col sep=&]{example6_convergenceData_d4.dat}{\tableDataB}
\pgfplotstablegetelem{1}{exponentCoeff}\of\tableDataB
\edef\exponentCoeffb{\pgfplotsretval}
\pgfplotstablegetelem{1}{exponentFit}\of\tableDataB
\edef\exponentFitb{\pgfplotsretval}

\pgfplotstablecreatecol[copy column from table={\tableDataB}{[index] 3}] {par1} {\tableDataA}
\pgfplotstablecreatecol[copy column from table={\tableDataB}{[index] 4}] {par2} {\tableDataA}

\makeatletter
\def\pgfmath@error#1#2{\vrule depth 1in}

\def\numorNaN#1{%
    \setbox0\hbox{#1}%
    \ifdim\dp0<1in\box0 \else\makebox[11em]{--}\fi}
\def\bracketorNaN#1{%
    \setbox0\hbox{#1}%
    \ifdim\dp0<1in(\box0 )\else\null\fi}

\begin{table}[htb]
    \centering
    \caption{Degree 3 ($d = 3$) and degree 4 ($d = 4$) future energy function approximations for the finite element beam as the mesh is refined.}
    \pgfplotstabletypeset[ 
        every head row/.style={before row=\toprule,after row=\midrule},
        every last row/.style={after row=\bottomrule},
        col sep=&,
        multicolumn names,
        columns/{n}/.style={column name=$n$},
        columns/mixed1/.style={string type, column type = {l},column name={$\mathcal{E}_3^+(\mathbf{x}_a)$ (CPU Sec)}},
        columns/mixed2/.style={string type, column type = {l},column name={$\mathcal{E}_4^+(\mathbf{x}_a)$ (CPU Sec)}},
        columns={{n},{mixed1},{mixed2}},
        create on use/mixed1/.style={
                create col/assign/.code={%
                        \edef\entry{\noexpand\numorNaN{\noexpand\pgfmathprintnumber[sci,sci zerofill, precision=6]{\thisrow{E_3^+(x_0)}}} \noexpand\bracketorNaN{\noexpand\pgfmathprintnumber[sci,sci zerofill, precision=1]{\thisrow{CPU-sec}}}}%
                        \pgfkeyslet{/pgfplots/table/create col/next content}\entry
                    }
            },
        create on use/mixed2/.style={
                create col/assign/.code={%
                        \edef\entry{\noexpand\numorNaN{\noexpand\pgfmathprintnumber[sci,sci zerofill, precision=6]{\thisrow{par2}}} \noexpand\bracketorNaN{\noexpand\pgfmathprintnumber[sci,sci zerofill, precision=1]{\thisrow{par1}}}}%
                        \pgfkeyslet{/pgfplots/table/create col/next content}\entry
                    }
            },
    ]{\tableDataA}
    \label{tab:example6_convergenceData}
\end{table}

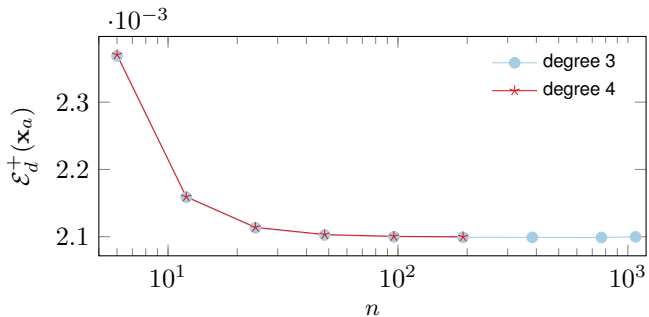
\begin{figure}
    \centering
    \begin{tikzpicture}
        \begin{semilogxaxis}[xlabel=$n$,
                ylabel=$\mathcal{E}_d^+(\mathbf{x}_a)$,
                width=\columnwidth, height=4.5cm,
                xmin=5, xmax=1.2e3,
                ylabel near ticks,
                legend pos=north east,
                legend style={draw=none,font=\scriptsize\sffamily},
                xlabel style={font=\small\sffamily},
                ylabel style={font=\small\sffamily},
            ]

            \addplot[color=Paired-A,mark=*,mark size=2] table[x={n}, y={E_3^+(x_0)},col sep=&] {example6_convergenceData_d3.dat};
            \addlegendentry{degree 3}

            \addplot[color=Paired-F,mark=star,mark size=2] table[x={n}, y={E_4^+(x_0)},col sep=&] {example6_convergenceData_d4.dat};
            \addlegendentry{degree 4}
        \end{semilogxaxis}
    \end{tikzpicture}
    \caption{Convergence w.r.t. $n$ of the future energy function evaluated at $\bx_a$ as the finite element mesh is refined.}
    \label{fig:example6_convergence_n}
\end{figure}

We also present the CPU time required to compute the energy approximations in \cref{tab:example6_convergenceData}
to investigate the scalability of the proposed algorithm as the number of elements is increased.
As can be seen in \cref{fig:example6_cpuScaling_n}, the algorithms do appear to scale roughly as $O(dn^{d+1})$ as predicted in \cref{sec:complexity}.
For low model dimensions, the scaling deviates due to the minimum time required to run the entire program, which is on the order of milliseconds.
Currently, the main hinderance to further scalability is memory usage rather than the mathematical operations themselves, which is not accounted for in a flop count computational complexity analysis.

\begin{figure}
    \centering
    \begin{tikzpicture}
        \begin{loglogaxis}[xlabel=$n$,
                ylabel=CPU sec,
                width=\columnwidth, height=4.5cm,
                xmin=5, xmax=1.2e3,
                ymin=1e-3, ymax=1e4,
                legend pos=south east,
                legend style={draw=none,font=\scriptsize\sffamily},
                xlabel style={font=\small\sffamily},
                ylabel style={font=\small\sffamily},
            ]
            \addplot[domain=5:2e3, samples=10, color=lightgray, thick, forget plot] {6*10^(-8) * x^3} node [pos=.8, fill=white, sloped] {$n^{3}$};
            \addplot[domain=5:6e2, samples=10, color=lightgray, thick, forget plot] {6*10^(-8) * x^4} node [pos=.8, fill=white, sloped] {$n^{4}$};
            \addplot[domain=5:6e2, samples=10, color=lightgray, thick, forget plot] {10^(-7) * x^5} node [pos=.6, fill=white, sloped] {$n^{5}$};

            \addplot[color=Paired-A,mark=*,mark size=2] table[x={n}, y={CPU-sec},col sep=&] {example6_convergenceData_d3.dat};
            \addlegendentry{degree 3}

            \addplot[color=Paired-F,mark=star,mark size=2] table[x={n}, y={CPU-sec},col sep=&] {example6_convergenceData_d4.dat};
            \addlegendentry{degree 4}
        \end{loglogaxis}
    \end{tikzpicture}
    \caption{Scaling of CPU time as $n$ increases for $d=3,4$. The computational complexity scales as $O(dn^{d+1})$, as predicted by the flop count.}
    \label{fig:example6_cpuScaling_n}
\end{figure}
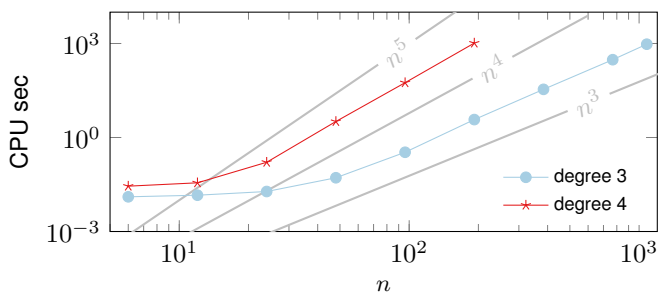

Next, we fix the size of the model $n$ while increasing the degree of the polynomial approximation to the energy functions in \cref{tab:example6_convergence} for the two initial conditions $\bx_a$ and $\bx_b$.
The initial condition $\bx_a$ corresponding to $x_0 = 0.01$ is closer to the equilibrium at the origin, whereas $\bx_b$ corresponds to $x_0 = 0.1$ and is further from the origin.
Intuitively then, as we saw in \cref{sec:example1}, since $\bx_b$ is further from the equilibrium at the origin, we expect it to require a higher-degree energy function approximation in order for the energy to be resolved properly.
We see this in the second and third columns in \cref{tab:example6_convergence} and plotted in \cref{fig:example6_convergence}, where the energy for the initial condition $\bx_a$ converges with just a degree $d=3$ approximation, whereas $\bx_b$ requires a degree $d=4$ approximation in order to converge.

\pgfplotstableread[col sep=&]{example6_convergenceData_e3_biggerIC.dat}\dataB
\pgfplotstableread[col sep=&]{example6_convergenceData_e3.dat}\dataA

\pgfplotstablecreatecol[copy column from table={\dataB}{[index] 1}] {par1} {\dataA}
\pgfplotstablecreatecol[copy column from table={\dataB}{[index] 2}] {par2} {\dataA}

\begin{table}[htb]
    \centering
    \caption{Future energy function approximation for the finite element beam with 3 elements ($n=18$) as the energy function is approximated to higher-orders $d$ for initial conditions $\bx_a,\bx_b$.}
    \pgfplotstabletypeset[ 
    every head row/.style={before row=\toprule,after row=\midrule},
    every last row/.style={after row=\bottomrule},
    col sep=&,
    multicolumn names,
    columns/{d}/.style={
            column name=$d$},
    columns/{CPU-sec-2}/.style={
            sci,sci zerofill,sci sep align,
            precision=2,
            column name=CPU sec a},
    columns/{E_d^+(x_0)}/.style={
    sci,sci zerofill,sci sep align,
    precision=6,
    column name=$\mathcal{E}_d^+(\mathbf{x}_a)$},
    columns/{par1}/.style={
            sci,sci zerofill,sci sep align,
            precision=2,
            column name=CPU sec b},
    columns/{par2}/.style={
            sci,sci zerofill,sci sep align,
            precision=6,
            column name=$\mathcal{E}_d^+(\mathbf{x}_b)$},
    columns/mixed1/.style={string type, column type = {l},column name={$\mathcal{E}_d^+(\mathbf{x}_a)$ (CPU Sec)}},
    columns/mixed2/.style={string type, column type = {l},column name={$\mathcal{E}_d^+(\mathbf{x}_b)$ (CPU Sec)}},
    columns={{d},{mixed1},{mixed2}},
    create on use/mixed1/.style={
            create col/assign/.code={%
                    \edef\entry{\noexpand\numorNaN{\noexpand\pgfmathprintnumber[sci,sci zerofill, precision=6]{\thisrow{E_d^+(x_0)}}} \noexpand\bracketorNaN{\noexpand\pgfmathprintnumber[sci,sci zerofill, precision=1]{\thisrow{CPU-sec-2}}}}%
                    \pgfkeyslet{/pgfplots/table/create col/next content}\entry
                }
        },
    create on use/mixed2/.style={
            create col/assign/.code={%
                    \edef\entry{\noexpand\numorNaN{\noexpand\pgfmathprintnumber[sci,sci zerofill, precision=6]{\thisrow{par2}}} \noexpand\bracketorNaN{\noexpand\pgfmathprintnumber[sci,sci zerofill, precision=1]{\thisrow{par1}}}}%
                    \pgfkeyslet{/pgfplots/table/create col/next content}\entry
                }
        },
    ]{\dataA}
    \label{tab:example6_convergence}
\end{table}

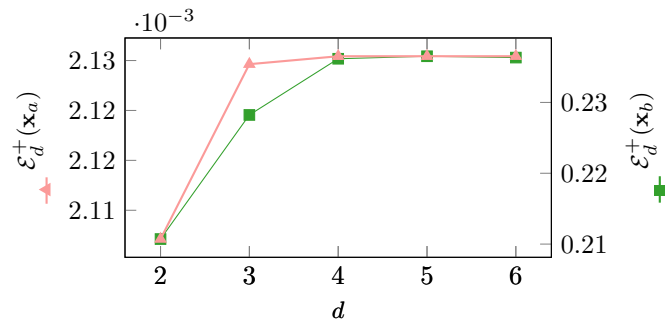
\begin{figure}[htb]
    \centering
    \begin{tikzpicture}
        \begin{axis}[xlabel=$d$,
                ylabel={ \tikz[baseline=-.5ex]\draw[Paired-D,thick] (0,0) -- (10pt,0) node[pos=.25pt,fill=Paired-D,rectangle,scale=0.65] {}; $\mathcal{E}_d^+(\mathbf{x}_b)$},
                width=7.25cm, height=4.5cm,
                yticklabel pos=right, ytick pos=right,
                xlabel style={font=\small\sffamily},
                ylabel style={font=\small\sffamily},
            ]
            \addplot
            [color=Paired-D,mark=square*]
            table [x=d, y={E_d^+(x_0)},col sep=&] {example6_convergenceData_e3_biggerIC.dat};
        \end{axis}
        \begin{axis}[
                xlabel=$d$,
                ylabel={ \tikz[baseline=-.5ex]\draw[Paired-E,thick] (0,0) -- (10pt,0) node[pos=.725pt,fill=Paired-E,regular polygon, regular polygon sides=3,scale=0.35] {}; $\mathcal{E}_d^+(\mathbf{x}_a)$},
                width=7.25cm, height=4.5cm,
                yticklabel pos=left, ytick pos=left,
                xlabel style={font=\small\sffamily},
                ylabel style={font=\small\sffamily},
            ]
            \addplot
            [color=Paired-E,mark=triangle*,thick,mark size=2]
            table [x=d, y={E_d^+(x_0)},col sep=&] {example6_convergenceData_e3.dat};
        \end{axis}
    \end{tikzpicture}
    \caption{Future energy function convergence as the degree $d$ increases for initial conditions $\bx_a,\bx_b$.
        While a degree 3 approximation is sufficient to converge for $\bx_a$, a degree 4 approximation is necessary for $\bx_b$. }
    \label{fig:example6_convergence}
\end{figure}

If the initial condition is too far from the origin, the polynomial approximation is known to diverge to either $\pm \infty$, as was the case with the previous examples.

\section{Conclusion}\label{sec:conclusion}
In this work, we proposed a Kronecker product-based approach to
computing nonlinear balanced truncation energy functions
for control-affine systems.
The three distinct improvements over the approach in \cite{Kramer2024} are the ability to handle arbitrary as opposed to only quadratic drift nonlinearity, the ability to handle polynomial inputs, and the ability to handle polynomial outputs.
As part of making this approach scalable, we derived  explicit structured formulas for the coefficients in the Taylor-series expansions of the energy functions.
Moreover, with the proposed efficient implementation, we showed that the approach scales with a cost $O(dn^{d+1})$, as opposed to the $O(n^{3d})$ cost of a naive implementation.
This was made possible by exploiting symmetry in the coefficients to compactly write the contributions of the many polynomial terms.
We provided numerical results to demonstrate that this approach can compute energy functions for systems with state dimensions up to $n=1080$ on a workstation computer.

The numerical examples further showed that the proposed method can accurately approximate non-quadratic energy functions which are inherent to nonlinear systems.
The results also demonstrated that, since the contribution in this work enables computing the true Taylor expansions of the energy functions, elementary results regarding Taylor expansions can be used to understand characteristics of the approximations computed herein.
The polynomial approximations will converge locally, but care must be taken to ensure that states remain in the function's region of convergence.
We also introduced a new benchmark problem based on a finite element discretization of a nonlinear beam; we used this model to study the scalability of the proposed algorithms.

The next steps for this work involve control and model reduction using the proposed energy function approximations, which would enable reduced-order model and controller design for output-feedback problems.
This requires developing scalable algorithms to compute nonlinear balancing transformations.
Furthermore, in this work, a direct solver was used to compute exact solutions for the energy function polynomial coefficients; the possibility of using iterative solvers and other approximations is of interest in order to a) speed up the computations further, and b) reduce memory requirements.
This can enable computing energy functions and ultimately reduced-order models
in even higher state-space dimensions.


\bibliography{main}

\begin{thebibliography}{10}

\bibitem{Albrekht1961}
E.~G. Al'brekht.
\newblock On the optimal stabilization of nonlinear systems.
\newblock {\em Journal of Applied Mathematics and Mechanics}, 25(5):1254--1266,
  Jan. 1961.

\bibitem{Almubarak2019}
H.~Almubarak, N.~Sadegh, and D.~G. Taylor.
\newblock Infinite horizon nonlinear quadratic cost regulator.
\newblock In {\em 2019 American Control Conference ({ACC})}, pages 5570--5575,
  July 2019.

\bibitem{Beard1997}
R.~W. Beard, G.~N. Saridis, and J.~T. Wen.
\newblock Galerkin approximations of the generalized hamilton-jacobi-bellman
  equation.
\newblock {\em Automatica}, 33(12):2159--2177, Dec. 1997.

\bibitem{Benner2024}
P.~Benner and P.~Goyal.
\newblock Balanced truncation for quadratic-bilinear control systems.
\newblock {\em Advances in Computational Mathematics}, 50(88), Aug. 2024.

\bibitem{NLBalancing2023}
J.~Borggaard and N.~A. Corbin.
\newblock {NL}balancing repository.
\newblock Available online: \url{https://github.com/cnick1/NLbalancing}, June
  2023.

\bibitem{Borggaard2020}
J.~Borggaard and L.~Zietsman.
\newblock The quadratic-quadratic regulator problem: approximating feedback
  controls for quadratic-in-state nonlinear systems.
\newblock In {\em 2020 American Control Conference (ACC)}, pages 818--823, July
  2020.

\bibitem{Borggaard2021}
J.~Borggaard and L.~Zietsman.
\newblock On approximating polynomial-quadratic regulator problems.
\newblock {\em {IFAC}-{PapersOnLine}}, 54(9):329--334, 2021.

\bibitem{Borovykh2022}
A.~Borovykh, D.~Kalise, A.~Laignelet, and P.~Parpas.
\newblock Data-driven initialization of deep learning solvers for
  {H}amilton-{J}acobi-{B}ellman {PDEs}.
\newblock {\em {IFAC}-{PapersOnLine}}, 55(30):168--173, Nov. 2022.

\bibitem{Breiten2018}
T.~Breiten, K.~Kunisch, and L.~Pfeiffer.
\newblock Numerical study of polynomial feedback laws for a bilinear control
  problem.
\newblock {\em Mathematical Control \& Related Fields}, 8(3):557--582, 2018.

\bibitem{Brewer1978}
J.~Brewer.
\newblock Kronecker products and matrix calculus in system theory.
\newblock {\em {IEEE} Transactions on Circuits and Systems}, 25(9):772--781,
  Sept. 1978.

\bibitem{Chen2019}
M.~Chen and D.~Kressner.
\newblock Recursive blocked algorithms for linear systems with {K}ronecker
  product structure.
\newblock {\em Numerical Algorithms}, 84(3):1199--1216, Sept. 2019.

\bibitem{Cimen2012}
T.~{\c{C}}imen.
\newblock Survey of state-dependent {R}iccati equation in nonlinear optimal
  feedback control synthesis.
\newblock {\em Journal of Guidance, Control, and Dynamics}, 35(4):1025--1047,
  July 2012.

\bibitem{Condon2005}
M.~Condon and R.~Ivanov.
\newblock Nonlinear systems {\textendash} algebraic {G}ramians and model
  reduction.
\newblock {\em {COMPEL} - The international journal for computation and
  mathematics in electrical and electronic engineering}, 24(1):202--219, Mar.
  2005.

\bibitem{Dolgov2021}
S.~Dolgov, D.~Kalise, and K.~K. Kunisch.
\newblock Tensor decomposition methods for high-dimensional
  {H}amilton-{J}acobi-{B}ellman equations.
\newblock {\em {SIAM} Journal on Scientific Computing}, 43(3):A1625--A1650,
  Jan. 2021.

\bibitem{Falcone2016}
M.~Falcone and R.~Ferretti.
\newblock Numerical methods for {H}amilton-{J}acobi type equations.
\newblock In {\em Handbook of Numerical Analysis}, pages 603--626. Elsevier,
  2016.

\bibitem{Fujimoto2010}
K.~Fujimoto and J.~M.~A. Scherpen.
\newblock Balanced realization and model order reduction for nonlinear systems
  based on singular value analysis.
\newblock {\em {SIAM} Journal on Control and Optimization}, 48(7):4591--4623,
  Jan. 2010.

\bibitem{Fujimoto2008a}
K.~Fujimoto and D.~Tsubakino.
\newblock Computation of nonlinear balanced realization and model reduction
  based on {T}aylor series expansion.
\newblock {\em Systems \& Control Letters}, 57(4):283--289, Apr. 2008.

\bibitem{Garrard1972}
W.~L. Garrard.
\newblock Suboptimal feedback control for nonlinear systems.
\newblock {\em Automatica}, 8(2):219--221, Mar. 1972.

\bibitem{Gray2006}
W.~S. Gray and E.~I. Verriest.
\newblock Algebraically defined gramians for nonlinear systems.
\newblock In {\em 2006 45th {IEEE} Conference on Decision and Control}. {IEEE},
  Dec. 2006.

\bibitem{Gugercin2004}
S.~Gugercin and A.~C. Antoulas.
\newblock A survey of model reduction by balanced truncation and some new
  results.
\newblock {\em International Journal of Control}, 77(8):748--766, May 2004.

\bibitem{Henderson1981}
H.~V. Henderson and S.~R. Searle.
\newblock The vec-permutation matrix, the vec operator and {K}ronecker
  products: a review.
\newblock {\em Linear and Multilinear Algebra}, 9(4):271--288, Jan. 1981.

\bibitem{Kalise2018}
D.~Kalise and K.~Kunisch.
\newblock Polynomial approximation of high-dimensional
  {H}amilton-{J}acobi-{B}ellman equations and applications to feedback control
  of semilinear parabolic {PDE}s.
\newblock {\em SIAM Journal on Scientific Computing}, 40(2):A629--A652, 2018.

\bibitem{Kawano2017}
Y.~Kawano and J.~M.~A. Scherpen.
\newblock Model reduction by differential balancing based on nonlinear {H}ankel
  operators.
\newblock {\em {IEEE} Transactions on Automatic Control}, 62(7):3293--3308,
  July 2017.

\bibitem{Kramer2024}
B.~Kramer, S.~Gugercin, J.~Borggaard, and L.~Balicki.
\newblock Scalable computation of energy functions for nonlinear balanced
  truncation.
\newblock {\em Computer Methods in Applied Mechanics and Engineering},
  427:117011, July 2024.

\bibitem{Krener2008}
A.~J. Krener.
\newblock Reduced order modeling of nonlinear control systems.
\newblock In {\em Analysis and Design of Nonlinear Control Systems}, pages
  41--62. Springer Berlin Heidelberg, 2008.

\bibitem{Krener2019}
A.~J. Krener.
\newblock Nonlinear {S}ystems {T}oolbox.
\newblock Available on request to ajkrener@nps.edu, 2019.

\bibitem{Lukes1969}
D.~L. Lukes.
\newblock Optimal regulation of nonlinear dynamical systems.
\newblock {\em {SIAM} Journal on Control}, 7(1):75--100, Feb. 1969.

\bibitem{Magnus2019}
J.~R. Magnus and H.~Neudecker.
\newblock {\em Matrix differential calculus with applications in statistics and
  econometrics}.
\newblock Wiley, third edition, Feb. 2019.

\bibitem{Marsden1998a}
J.~Marsden and A.~Weinstein.
\newblock {\em Calculus {II}}.
\newblock Springer, 1998.

\bibitem{Moore1981}
B.~C. Moore.
\newblock Principal component analysis in linear systems: {Controllability},
  observability, and model reduction.
\newblock {\em IEEE Transactions on Automatic Control}, 26:17--32, Feb. 1981.

\bibitem{Mullis1976}
C.~Mullis and R.~Roberts.
\newblock Synthesis of minimum roundoff noise fixed point digital filters.
\newblock {\em {IEEE} Transactions on Circuits and Systems}, 23(9):551--562,
  Sept. 1976.

\bibitem{Parrilo2000}
P.~A. Parrilo.
\newblock {\em Structured semidefinite programsand semialgebraic geometry
  methodsin robustness and optimization}.
\newblock PhD thesis, California Institute of Technology, May 2000.

\bibitem{Reddy2004}
J.~N. Reddy.
\newblock {\em An introduction to nonlinear finite element analysis}.
\newblock Oxford University Press, 2004.

\bibitem{Scherpen1993}
J.~M.~A. Scherpen.
\newblock Balancing for nonlinear systems.
\newblock {\em Systems \& Control Letters}, 21(2):143--153, Aug. 1993.

\bibitem{Scherpen1994a}
J.~M.~A. Scherpen.
\newblock {\em Balancing for nonlinear systems}.
\newblock PhD thesis, University of Twente, 1994.

\bibitem{Scherpen1996}
J.~M.~A. Scherpen.
\newblock {$\mathcal{H}_\infty$} balancing for nonlinear systems.
\newblock {\em International Journal of Robust and Nonlinear Control},
  6(7):645--668, Aug. 1996.

\bibitem{Scherpen1994}
J.~M.~A. Scherpen and A.~J. Van Der~Schaft.
\newblock Normalized coprime factorizations and balancing for unstable
  nonlinear systems.
\newblock {\em International Journal of Control}, 60(6):1193--1222, Dec. 1994.

\bibitem{VanLoan2000}
C.~F. Van~Loan.
\newblock The ubiquitous {K}ronecker product.
\newblock {\em Journal of Computational and Applied Mathematics},
  123(1-2):85--100, Nov. 2000.

\end{thebibliography}
\bibliographystyle{abbrv} 

\begin{IEEEbiography}[{\includegraphics[width=1in,height=1.25in,clip,keepaspectratio]{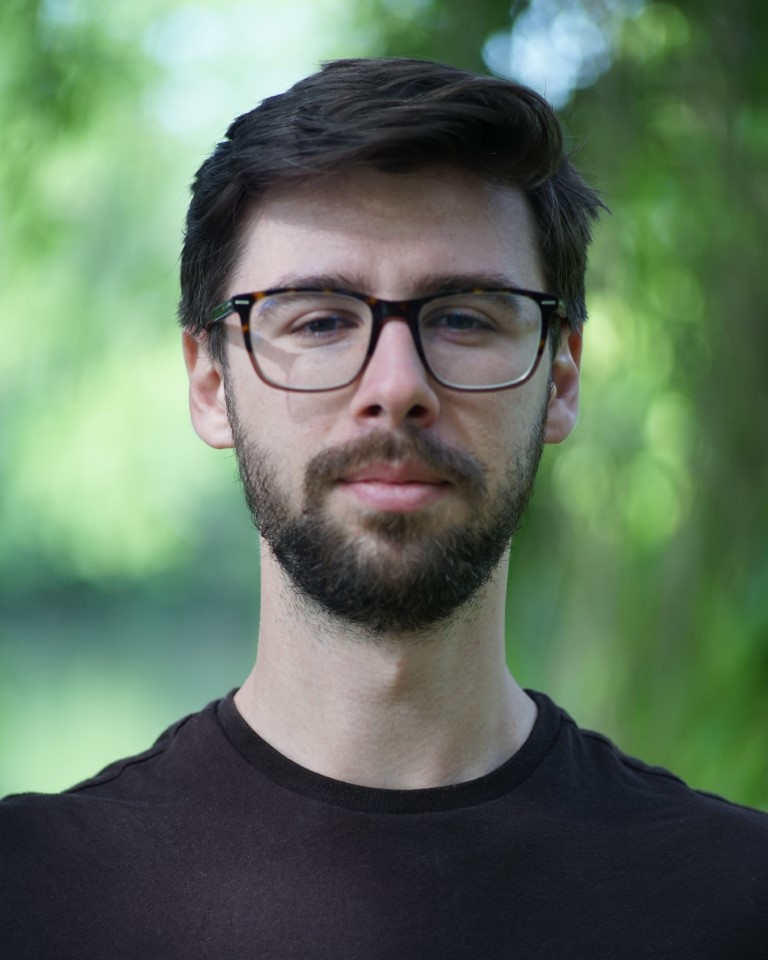}}]{Nicholas A. Corbin} was born in Portsmouth, VA on October 5, 1997. He received the B.S. and M.S. degrees in engineering science and mechanics from Virginia Tech, Blacksburg, Virginia in 2019 and 2021, respectively.

    Since 2021, he has been a Ph.D. student at the University of California San Diego.
    His main research interests are model reduction, nonlinear control, and mechanical vibrations.
\end{IEEEbiography}

\begin{IEEEbiography}[{\includegraphics[width=1in,height=1.25in,clip,keepaspectratio]{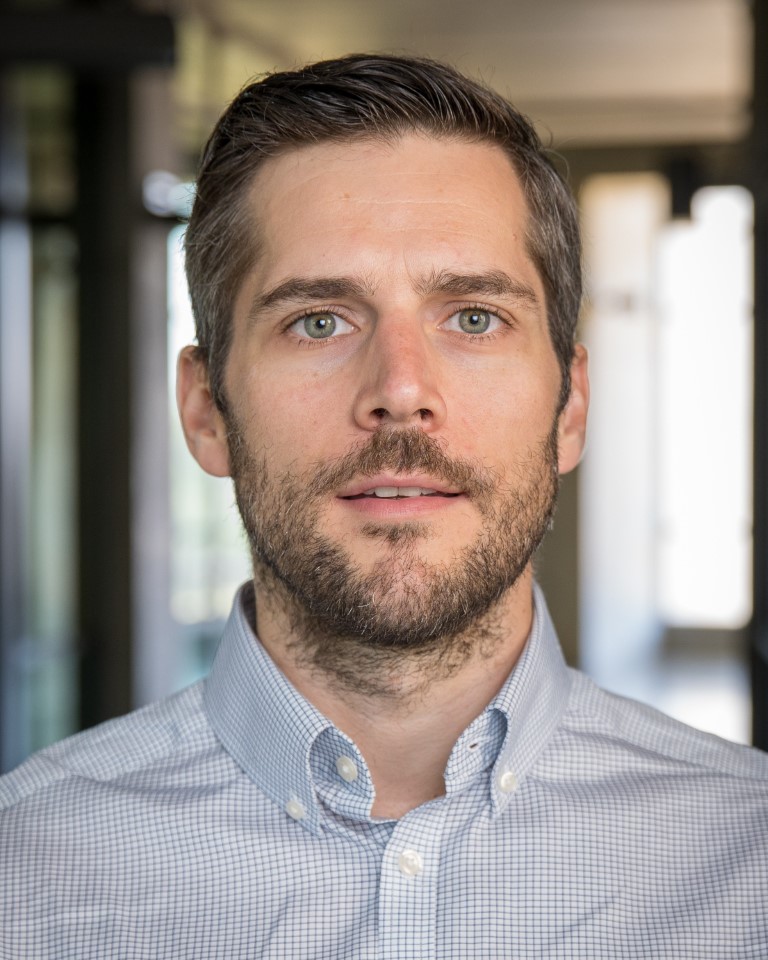}}]{Boris Kr\"{a}mer} received the M.S. and Ph.D. degrees in mathematics from Virginia Tech, Blacksburg, Virginia in 2011 and 2015, respectively.

    He is an Assistant Professor of Mechanical and Aerospace Engineering at the University of California San Diego, USA. Prior to that, he has been a Postdoctoral Scholar at the Massachusetts Institute of Technology from 2015-2019. His main research interests are model reduction, data-driven modeling, (multifidelity) uncertainty quantification, and design under uncertainty.
    Dr. Kramer is a member of SIAM and a Senior Member of AIAA. He received the National Science Foundation Early CAREER Award in Dynamics, Control and System Diagnostics in 2022 and the Department of Defense Newton Award in 2020. He presently is an Associate Editor for the SIAM/ASA Journal on Uncertainty Quantification.
\end{IEEEbiography}

\end{document}